\documentclass[12pt]{amsart}%
\usepackage{amsfonts}
\usepackage{epsfig}
\usepackage{graphicx}
\usepackage{amsmath}
\usepackage{amsfonts}
\usepackage{amssymb}
\usepackage{latexsym}
\usepackage{rotating}
\usepackage{pstricks, pst-node, pst-text, pst-3d,pst-eps}
\usepackage[arrow, matrix, curve, xdvi, dvips]{xy} 
\usepackage[text={5.35in,8.3in}]{geometry}
\usepackage{stmaryrd}
\usepackage{amsbsy}
\usepackage{bm}
\usepackage{dsfont}
\usepackage{mathrsfs}

\input{amssym.def}
\newtheorem{theorem}{Theorem}
\newtheorem{proposition}{Proposition}
\newtheorem{corollary}{Corollary}
\newtheorem{lemma}{Lemma}
\newtheorem{remark}{Remark}

\newtheorem{definition}{Definition}
\theoremstyle{remark}
\newtheorem{example}{Example}

\newcommand{\Id}{\text{\rm Id}}

\newcommand{\spn}{\text{\rm span}\,}

\newcommand{\GL}{\text{\rm GL}}

\newcommand{\Orth}{\text{\rm O}}
\newcommand{\Iso}{\text{\rm Iso}}
\newcommand{\End}{\text{\rm End}}
\newcommand{\Adm}{\text{\rm Adm}}
\newcommand{\Cl}{\text{\rm Cl}}
\newcommand{\U}{\text{\rm U}}
\newcommand{\Gr}{\text{\rm Gr}}
\newcommand{\Sp}{\text{\rm Sp}}

\newcommand{\Pin}{\text{\rm Pin}}

\newcommand{\diam}{\text{\rm diam}}

\newcommand{\Aut}{\text{\rm Aut}}

\newcommand{\g}{\mathfrak{g}}

\newcommand{\h}{\mathfrak{h}}

\newcommand{\R}{\mathbb{R}}
\newcommand{\K}{\mathbb{K}}
\newcommand{\CC}{\mathbb{C}}
\newcommand{\Q}{\mathbb{Q}}

\newcommand{\mc}{\mathcal }

\newcommand{\M}{\mathbb{M}}
\newcommand{\G}{\mathbb{G}}
\newcommand{\HH}{\mathbb{H}}
\newcommand{\graph}[1]{\mathrm{graph}\,(#1)}

\newcommand{\norm}[1]{\left\Vert{#1}\right\Vert}
\newcommand{\res}{\mathop{\hbox{\vrule height 7pt width .5pt depth 0pt \vrule height .5pt width 6pt depth 0pt}}\nolimits}

\usepackage{pdfsync}

\usepackage{color}

\usepackage{hyperref}

\begin{document}
\title[Modules]{Exceptional families of measures on Carnot groups}
\author[B.~Franchi, I.~Markina]{Bruno Franchi and Irina Markina}

\thanks{The work of the second author was partially supported by the project Pure Mathematics in Norway, funded by Trond Mohn Foundation and Troms\o\ Research Foundation.}
 

\address{Bruno Franchi. Department of Mathematics, University of Bologna,
Piazza di Porta S. Donato, 5
    40126 Bologna, Italy}
\email{bruno.franchi@unibo.it}

\address{Irina Markina. Department of Mathematics, University of Bergen, P.O.~Box 7803,
Bergen N-5020, Norway}
\email{irina.markina@uib.no}

\begin{abstract}
We study the families of measures on Carnot groups that have vanishing $p$-module, which we call $p$-exceptional families. We found necessary and sufficient condition for the family of intrinsic Lipschitz surfaces passing through a common point to be $p$-exceptional for $p\geq 1$. We described a wide class of $p$-exceptional intrinsic Lipschitz surfaces for $p\in(0,\infty)$. 
\end{abstract}
\maketitle


\tableofcontents


\section{Introduction and motivation}


 ``Negligible'' sets appear customarily in measure theory and stochastic theory, as the sets of measure zero, as 
well as the sets of vanishing $p$-capacity when dealing
with regularity issues for solutions of PDE, and as thin and polar sets in potential theory.
Sets of a family of measures having the so-called
$p$-module zero belong to this category of negligible or exceptional subsets of families of measures, see definitions in Section~\ref{sec:DefinitionModule}. 

The notion of a module of a family of curves or in another terminology extremal length originated in the theory of complex analytic functions as a conformal invariant~\cite{MR36841}, and later was widely used for the quasiconformal analysis and extremal problems of functional spaces~\cite{MR1677985, MR2788359, MR2068961,MR1238941,  MR440031,MR0120346}. B.~Fuglede in his seminal paper~\cite{MR97720} proposed to extend the notion of the module from families of curves to families of measures. He characterized the completion, with respect to $L^p$ norm, of some functional classes by using a family of surfaces having vanishing $p$-module. He also described some classes of systems of measures with vanishing module and related it to the potential theory. In spite of the fact that the definition of the module of a family of measures is given for an arbitrary measure space, most of the applications in~\cite{MR97720} were done for $\mathbb R^n$. 

The development of the analysis on metric measure spaces inspired us to look for examples of interesting systems of measures in a more general setting 
then the Euclidean space. We are not the first ones, just to name~\cite{MR3076803, MR1936925,MR2097161,MR1809341,MR1879250}. However, most of the preceding works were dealing with families of curves. Our main interest focuses on
families of (suitably defined) intrinsic surfaces on Carnot groups~\cite{MR2032504,MR3511465,MR3060706}. 

Carnot groups
are connected, simply connected, nilpotent Lie groups and are one of the most popular examples of metric measure spaces. 
Being endowed with a rich structure of translations and dilations, makes the Carnot groups
akin to Euclidean spaces.  
%
%
%
%
Euclidean spaces are commutative Carnot
groups, and, more precisely, the only commutative Carnot
groups.
The simplest but, at the same time,  non-trivial
instance of  non-abelian Carnot groups is provided by Heisenberg groups $H^n$.  

Carnot groups possess an intrinsic
metric, the so-called Carnot-Carath\'eodory metric ($cc$-distance), see for instance, \cite{MR2363343, MR1984849, MR1421823}.
It is also well known that non commutative Carnot groups, endowed with the
$cc$-distance, are not Riemannian manifolds because the
$cc$-distance makes them not locally Lipschitz equivalent to
Riemannian at any scale~\cite{Semmes}. The Carnot groups are particular instances of the so-called sub-Riemannian manifolds.

Though Carnot groups are analytic manifolds, the study of measures supported
on submanifolds (for instance the Hausdorff measures associated with
their $cc$-distance) cannot be reduced to the well established theory for submanifolds
of Euclidean spaces, since it has been clear for a long time that considering Euclidean 
regular submanifolds, even in Heisenberg groups, may be both too general and too 
restrictive, see~\cite{MR2124590} for a striking example related to the second instance. 
Through this paper, we shall rely
on the theory of intrinsic submanifolds in
Carnot groups that has been recently developed by making use of the notion of {\emph {intrinsic
graphs}}, see e.g.~\cite{FrSerSC07,MR3511465,MR2032504}. A discussion of
different alternatives leading to this notion can be found e.g. in~\cite{MR3511465},
together with the main properties of the most relevant instances, the so-called
intrinsic Lipschitz graphs.
Let us sketch this construction, restricting ourselves to stress the difficulties
arising when we want to extend the theory of $p$-modules from the Euclidean
setting to Carnot groups. For deep algebraic reasons, due to the non-commutativity of the group,
the most flexible notion of submanifold of a Carnot group is the counterpart of
the Euclidean notion of graph. However, the notion of intrinsic graph is not
a straightforward translation of the corresponding Euclidean notion, since Carnot
groups not always can be expressed as a direct product of subgroups. Because of that,
we argue as follows: an intrinsic graph inside $\G$ is associated with a decomposition of 
$\G$  as a  product $\G = \mathbb M \cdot \mathbb H$ of two homogeneous
{\emph {complementary subgroups}} $\mathbb M$, $\mathbb H$, see Section~\ref{sec:IntrinsicLipschitzsurface}. Then the intrinsic (left) graph of $f\colon \Omega\to \mathbb H$, where $\Omega$ is an open subset of $\mathbb M$, is the set
\begin{equation*}
\graph f =\{g\cdot f(g): g\in \Omega \}.
\end{equation*}
Another deep peculiarity of Carnot groups is the
poor structure of the isometry group preserving the grading structure of the their Lie algebras. 

The main results of the present work are formulated in Theorem~\ref{th:multipleDecomposition-1}, Section~\ref{sec:ExceptionalLipSurf}, where we show that quite a wide class of families of intrinsic Lipschitz surfaces, 
(sets which are locally intrinsic Lipschitz graphs of the same ``metric dimension'') has vanishing $p$-module for $p\in(0,1)$. We did not reach the full generality as in the Euclidean space due to the lack of knowledge about decompositions of an arbitrary Carnot groups into the product of two homogeneous subgroups. Another result contained in Theorem~\ref{prop:4}, Section~\ref{subsec:p1}, which is the sufficient condition for a family of surfaces passing through a common point to be $p$-exceptional. In order to find a necessary condition we construct a family of intrinsic Lipschitz graphs passing through one point by making use of the orthogonal Grassmannians on some specific 2-step Carnot groups, see Section~\ref{secOrthGrassmannians}. The construction of the orthogonal Grassmannians and the study of measures on them have an independent interest and, as to the knowledge of the authors, were not presented in the literature. Examples of exceptional families of measures that are not related to intrinsic Lipschitz graphs are contained in Examples 1 and 2 in Section~\ref{sec:Radon}.

We are trying to keep the paper as accessible as possible for a wider audience. The structure of the paper is visible from the Contents. 


\section{Carnot groups}


In the present section we establish
notation and collect the basic notions concerning Carnot groups and
their Lie algebras.


\subsection{General definition of Carnot groups}\label{sec:carnot-group}


A Carnot group $\G$ is a connected, simply connected Lie group whose Lie algebra $\mathfrak g$ of the left-invariant
vector fields is a graded stratified nilpotent  Lie algebra  of step $l$, i.e.  the Lie algebra $\mathfrak g$ satisfies:
$$
\mathfrak g=\oplus_{k=1}^{l}\mathfrak g_k,\quad
[\mathfrak g_1,\mathfrak g_{k}]=\mathfrak g_{k+1},\quad
\mathfrak g_{l+1}=\{0\}.
$$ 
We denote by $N=\sum_{k=1}^l\,\dim(\mathfrak g_k)$ the topological dimension of $\G$. The number $Q=\sum_{k=1}^l \,k\dim(\mathfrak g_k)$ is called the {\it homogeneous dimension} of the group $\G$. Since $\mathfrak g$ is nilpotent, the exponential map $\exp\colon\mathfrak g\to\G$ is a global diffeomorphism. 

One can identify the group $\G$ with $\mathbb R^N\cong \mathfrak g$ by making use of exponential coordinates of the first kind by the following procedure. We fix a basis 
\begin{equation}\label{eq:basis2ind}
X_{11},\ldots, X_{1\mathbf d_1},X_{21},\ldots,X_{2\mathbf d_2},\ldots,X_{l1},\ldots,X_{l\mathbf d_l},
\qquad \mathbf d_k=\dim(\mathfrak g_k),
\end{equation}
of the Lie algebra $\mathfrak g$ which is adapted to the stratification.
If $g\in \G$ and $V\in \mathfrak g$ are such that 
$$
g=\exp (V)=\exp \Big(\sum_{k=1}^l\sum_{j=1}^{\mathbf d_k}x_{kj}X_{kj}\Big),
$$ 
then (with a slight abuse of notations) we associate with the point $g\in \G$ a point $x\in\mathbb R^N$ having the following coordinates 
\begin{equation}\label{eq:coordinates}
g=(x_{11},\ldots, x_{1\mathbf d_1},x_{21},\ldots,x_{2\mathbf d_2},\ldots,x_{l1},\ldots,x_{l\mathbf d_l})=x.
\end{equation}
Thus, the identity $e\in\G$ is identified with the origin in $\mathbb R^N$ and the inverse $g^{-1}$ 
with $-x$.

The stratification $\mathfrak g=\oplus_{k=1}^{k=l}\mathfrak g_k$ of $\mathfrak g$ induces the one-parameter family 
$ \{\delta_\lambda \}_{\lambda>0}$ of automorphisms of $\mathfrak g$, where each $\delta_\lambda\colon \mathfrak g \to \mathfrak g$ is defined as 
\[
\delta_\lambda(X):=\lambda^k X,\quad \text{for all}\quad X\in \mathfrak g_k\quad \text{and}\quad \lambda>0.
\]
The exponential map allows to transfer these \emph{automorphisms  of $\mathfrak g$} to a family of \emph{automorphisms of the Lie group $\G$}: $\delta_\lambda^{\G}: \G \to \G$, so-called intrinsic dilations, defined  as
\[
\delta_\lambda^{\G}:=\mathrm{exp} \circ \delta_\lambda \circ \mathrm{exp}^{-1},\quad \text{for all}\quad \lambda>0.
\]
We keep  denoting by $\delta_\lambda: \G \to \G$ the intrinsic dilations, if no confusion arises.

In exponential coordinates, the group automorphism $\delta_{\lambda}\colon \G\to \G$ for $\lambda>0$ is written as
\begin{eqnarray}\label{eq:delta}
\delta_{\lambda}g&=&\delta_{\lambda}(x_{11},\ldots, x_{1\mathbf d_1},x_{21},\ldots,x_{2\mathbf d_2},\ldots,x_{l1},\ldots,x_{l\mathbf d_l})\nonumber
\\
&=&
(\lambda x_{11},\ldots, \lambda x_{1\mathbf d_1},\lambda^2x_{21},\ldots,\lambda^2 x_{2\mathbf d_2},\ldots,\lambda^lx_{l1},\ldots,\lambda^lx_{l\mathbf d_l}).
\end{eqnarray}


The group product on $\G$ written in coordinates~\eqref{eq:coordinates} has the form
\begin{equation}\label{legge di gruppo1}
x\cdot y=x+y+\mathcal Q(x,y),\quad
\text{for all}\quad x,y\in\mathbb R^N, 
\end{equation}
where $\mathcal Q=(\mathcal Q_1,\dots,\mathcal
Q_N):\mathbb R^N\times\mathbb R^N\to\mathbb R^N$ and each $\mathcal Q_k$ is a homogeneous
polynomial with respect to group dilations, see, for instance~\cite[Propositions 2.1 and 2.2]{MR1984849}.
 When the grading structure is not important, we will use the one-index notation
$$
X_1,\ldots,X_N,\qquad x_1,\ldots,x_N
$$
for the basis~\eqref{eq:basis2ind} of the Lie algebra $\g$ and for the coordinates on $\G$. The basis vectors of the Lie algebra $\g$ viewed as  left invariant vector fields $X_j$, $j=1,\ldots,N$ on $\G$ have polynomial coefficients and take the
form in the coordinate frame:
\begin{equation}\label{campi omogenei}
X_j=\partial_j+\sum_{i>j}^N q_{i,j}(x)\partial_i, \quad
\text{for}\; j=1,\dots,N,
\end{equation} where
$q_{i,j}(x)=\frac{\partial \mc Q_i}{\partial y_j}(x,y){|_{y=0}}$.
The vector fields $X_j$, $j=1,\ldots,\mathbf d_1$ are called horizontal and they are homogeneous of
degree $1$ with respect to the group dilation. 
Their span at  $q\in \G$ is called the horizontal vector space $H_q\G\subset T_q\G$. 


\subsubsection{Distance functions}\label{sec:Htypegroups}



In the present paper we will use the following distance functions on a Carnot
group $\G$ identified with $\mathbb R^N$ through exponential coordinates. 

\begin{itemize}
\item[($D_1$)]{The standard  Euclidean distance $d_{E}$ associated with
the Euclidean norm $|x|_{E}$: $d_E(x,y)=\sqrt{\sum_{i=1}^N(x_i-y_i)^2}$.}
\end{itemize}
However, such a distance is neither left-invariant under group translations, nor
1-homogeneous with respect to group dilations. Thus, let us introduce
further (not Lipschitz equivalent to $d_E$) distances enjoying these properties.

\begin{definition}\label{norminvar} Let $\G$ be a Carnot group.  A \emph{homogeneous norm} $\| \cdot\|$ is a continuous  function $\| \cdot\| :\,\G\to [0,+\infty)$  such that
\begin{equation}
\begin{split}\label{norm1}
&\| p\| =0\quad\text{if and only if}\quad p=0\,;\\
&\|p^{-1}\| =\| p\|,\quad \|\delta_{\lambda}(p)\| =\lambda \|p\|\quad \text{for all} \quad p\in\G \quad \text{and} \quad \lambda>0;\\
&\| {p\cdot q}\|\le\,\| p\| +\,\| q\|\qquad\text{for all}\quad p,q\in\G. 
\end{split}
\end{equation}
\end{definition}

\begin{remark}
A homogeneous norm $\| \cdot\|$ induces a  homogeneous left invariant distance in $\G$ as follows:
\begin{equation}\label{dist01}
d(p,q):=
d(q^{-1}\cdot p,0):=\|q^{-1}\cdot p\| \quad \text{for all} \quad p,q\in \G.
\end{equation}
\end{remark}

\begin{itemize}
\item[($D_2$)] {On any Carnot group $\G$
there exists a homogeneous norm $\|\cdot\|_{\G}$ that is smooth away of the origin and induces a distance $d_{\G}(x,y):=\|y^{-1}\cdot x\|_{\G}$ on $\G$, see~\cite[Page 638]{MR1232192}. 
}
\item[($D_3$)]{We also use the homogeneous norm $\|\cdot\|_{H}$:
$$
\|x\|_{H}=\max\{ \epsilon_1 \|\mathbf x_1\|_{E}, \epsilon_2 \|\mathbf x_2\|^{1/2}_{E},\ldots, \epsilon_l\|\mathbf x_l\|^{1/l}_E\},
$$
where $\mathbf x_k=(x_{k1},\ldots,x_{k\mathbf d_k})\in\g_k$, $\|\mathbf \cdot \|_E$ is the Euclidean norm, making the adapted basis~\eqref{eq:basis2ind} orthonormal. The suitable constants
$\epsilon_1,\dots,\epsilon_l$ are positive, see~\cite[Theorem 5.1]{MR1984849}. The 
induced distance is $d_{H}(x,y):=\|y^{-1}\cdot x\|_{H}$;
}
\item[($D_4$)]{The Carnot-Carath\'eodory distance $d_{cc}(x,y)$ which is induced by the Euclidean scalar product $\langle .\,,.\rangle_{\g_1}$ on $\g_1$, making the horizontal vector fields $X_j$, $j=1,\ldots,\mathbf d_1$ orthonormal~\cite{MR3971262, MR1421823}.}
\end{itemize}

%

The distances defined in $(D_2)-(D_4)$ are Lipschitz equivalent, since they are invariant under the left translation on $\G$ and are homogeneous functions 
of degree 1 with respect to dilation~\eqref{eq:delta}. We denote by $d_{\rho}$ any of the distances mentioned in $(D_2)-(D_4)$. Then the above observation and~\cite[Corollaries 5.15.1 and 5.15.2]{MR2363343}  imply:
\begin{proposition}\label{compareuclidtop} Let $\G$ be a   Carnot group of step $l$. Then 
\begin{itemize}
\item[(i)]   a set $A\subset \G$ is $d_{\rho}$-bounded if and only if it is $d_E$-bounded;
\item[(ii)]  
for any bounded set $A\subset\G$  there is $C_A>0$ such that
\[
C_A^{-1}\,d_E(x,y)\le\,d_{\rho}(x,y)\le\,C_A \,d_E(x,y)^{1/l}
\]
for all $x,y\in A$;
\item[(iii)] the topologies induced by $d_{\rho}$ and $d_E$ coincide.
\end{itemize}
\end{proposition}


\subsubsection{Measures on the Carnot groups}\label{sec:Htypegroups}


The pushforward of $N$-dimensional Lebesgue measure $\mathcal L^N$ on $\mathbb R^n$ to the group $\G$ under the exponential map is the Haar
measure $\mathbf g_{\G}$ on the group $\G$. Hence if $E\subset\mathbb R^N\cong \G$ is measurable,
then $ \mathcal L^N(x\cdot E)=\mathcal L^N(E\cdot x)= \mathcal L^N(E)$ for all $x\in\G$.
Moreover, if $\lambda>0$ then $\mathcal
L^N(\delta_\lambda(E))=\lambda^Q \mathcal L^N(E)$.

Recall the definition of the Hausdorff measure in a metric space $(X,\rho)$.
For all sets $E, E_i\subseteq X$, closed balls $B_{\rho}(x_i, r_i)$, real numbers $m\in [0,\infty)$, and $\delta >0$ one writes
\[
\begin{aligned}
\mathcal H_{\rho,\delta}^m(E)&:=\inf \Big\{\sum_i\diam(E_i)^m:\ E\subset \cup_i E_i,\,   \diam(E_i)\leq \delta \Big\},
\\
\mathcal S_{\rho,\delta}^m(E)&:=\inf \Big\{\sum_i \diam(E_i)^m:\  E\subset \cup_i B_{\rho}(x_i, r_i), \diam(B_{\rho}(x_i, r_i))\leq \delta  \Big\},\\
\end{aligned}
\]
where we assume $\diam(E_i)^0=1$ for $E_i\neq \emptyset$, and $\mathcal H_{\rho,\delta}^m(\emptyset)=\mathcal S_{\rho,\delta}^m(\emptyset)=0$. 
 For all $E\subseteq X$ and $m\in [0,\infty)$ the \emph{$m$-Hausdorff measure} $\mathcal H_\rho^m(E)$ and the \emph{spherical $m$-Hausdorff measure} $\mathcal S_{\rho}^m(E)$  are defined respectively
 as
 \[
 \mathcal H_\rho^m(E)=\lim_{\delta\to 0}\mathcal H_{\rho,\delta}^m(E),\qquad
 \mathcal S_\rho^m(E)=\lim_{\delta\to 0}\mathcal S_{\rho,\delta}^m(E).
\]
Both  $\mathcal H_\rho^m$ and $ \mathcal S_\rho^m$ are Borel regular measures.

When $\G$ is a Carnot group considered as a metric space $(\G,d_\rho)$, where $d_\rho$ is one of the distances $(D_1)-(D_4)$, we denote by  $\mathcal H_{d_\rho}^{\bf d_m}$ and $ \mathcal S_{d_\rho}^{\bf d_m}$ the ${\bf d_m}$-dimensional Hausdorff and spherical Hausdorff measures associated with the distance $d_\rho$, respectively.
The measures $\mathcal H_{d_\rho}^{\bf d_m}$ and $\mathcal S_{d_{\rho'}}^{\bf d_m}$, where $d_\rho$ and $d_{\rho'}$ are the distance functions of types $(D_2)-(D_4)$, satisfy
$$
c \mathcal H_{d_\rho}^{\bf d_m}(E)\leq \mathcal H_{d_{\rho'}}^{\bf d_m}(E)\leq C\mathcal H_{d_\rho}^{\bf d_m}(E),
\qquad
k \mathcal H_{d_\rho}^{\bf d_m}(E)\leq \mathcal S_{d_{\rho'}}^{\bf d_m}(E)\leq K\mathcal H_{d_\rho}^{\bf d_m}(E),
$$
for some positive constants $c,k,C,K$ and a set $E\subset \mathbb G$. The same is true if we interchange $\mathcal H_{d_\rho}^{\bf d_m}$ and $\mathcal S_{d_\rho}^{\bf d_m}$.
Finally
\begin{equation}\label{eq:QNequivalence}
\tilde c \mathcal H_{d_\rho}^Q(E)\leq \mathcal L^N(E)\leq \tilde C\mathcal H_{d_\rho}^Q(E),\quad E\subset \mathbb G,
\end{equation}
for some positive constants $\tilde c,\tilde C$, the homogeneous dimension $Q$, and the topological dimension $N$ of the Carnot group.

\subsubsection{$H$-type Lie groups}\label{sec:Htypegroups}


One of the core examples for the present paper will be $H$-type Lie groups, that are particular examples of 2-step Carnot groups.
Consider a real Lie algebra $(\g,[.\,,.],\langle.\,,.\rangle_{\mathbb R})$ with the underlying vector space $\g=\g_1\oplus \g_2$, where the decomposition is orthogonal with respect to the inner product $\langle.\,,.\rangle_{\mathbb R}$, and $\g_2$ is the center of the Lie algebra $\g$. The inner product space $(\g_2, \langle.\,,.\rangle_{\mathbb R})$, where $\langle.\,,.\rangle_{\mathbb R}$ is the restriction of the scalar product to the subspace $\g_2\subset \g$, generates the Clifford algebra $\Cl(\g_2, \langle.\,,.\rangle_{\mathbb R})$. The Clifford algebra $\Cl(\g_2, \langle.\,,.\rangle_{\mathbb R})$ admits a representation on the vector space $\g_1$:
$$
J\colon \Cl(\g_2, \langle.\,,.\rangle_{\mathbb R})\to\End(\g_1).
$$
We use the notation $J_z$, $z\in\g_2$, for the value of the map $J$ restricted to the vector space $\g_2\subset \Cl(\g_2, \langle.\,,.\rangle_{\mathbb R})$. 
From the definition of the Clifford algebra we have 
\begin{equation}\label{eq:J2}
J_z^2=-\langle z,z\rangle_{\mathbb R}\Id_{\g_1},\quad z\in\g_2.
\end{equation}

The Lie algebra $(\g,[.\,,.],\langle.\,,.\rangle_{\mathbb R})$ is called of $H$-type if 
\begin{equation}\label{eq:commutator}
\langle J_z u,v\rangle_{\mathbb R}=\langle z,[u,v]\rangle_{\mathbb R}, \quad z\in \g_2,\ \ u,v\in\g_1.
\end{equation}
An $H$-type Lie group is a connected simply connected Lie group whose Lie algebra is an $H$-type Lie algebra $(\g,[.\,,.],\langle.\,,.\rangle_{\mathbb R})$. 


\subsubsection{The Heisenberg group}\label{sec:Heis-R}


The $n$-th Heisenberg group $H^n$ is diffeomorphic to $\mathbb R^{2n+1}$ and is the simplest example of $H$-type Lie group. Its
$(2n+1)$-dimensional Lie algebra $\h^n_{\R}$ has one dimensional center $\h_2$. Let $\h_1$ be the $2n$-dimensional orthogonal complement to $\h_2$ with respect to an inner product $\langle.\,,.\rangle$ on $\h^n_{\R}$.
We choose an orthonormal basis 
\begin{equation}\label{eq:HeisFirstTime}
X_1,\ldots, X_n, Y_1,\ldots, Y_n\ \ \text{for}\ \ \h_1\ \ \text{and}\ \ \epsilon\ \ \text{for}\ \ \h_2,
\end{equation}
satisfying the commutation relations
\begin{equation}\label{eq:HeisCommFirstTime}
[X_j,Y_i]=\delta_{ji}\epsilon,\quad  [X_j,X_i] =[Y_j,Y_i]=0.
\end{equation}
Then the map $J_{\epsilon}$ defined in~\eqref{eq:commutator} satisfies
$$
J_{\epsilon}^2=-\Id_{\h_1},\quad J_{\epsilon}(X_i)=Y_i,\quad J_{\epsilon}(Y_i)=-X_i.
$$

\section{Module of a family of measures}\label{sec:DefinitionModule}


We start from the explaining the notion of a $p$-module of a system of measures, that B.~Fuglede introduced 
in his celebrated paper~\cite{MR97720}. 
Let $(X,\mathfrak{M}, m)$ be an abstract measure space with a fixed basic measure 
$m\colon\mathfrak{M}\to[0,+\infty]$ defined on a $\sigma$-algebra $\mathfrak{M}$ of subsets of $X$. We denote by $\mathbf{M}$ the system of all measures on $X$, whose 
domains of definition contain $\mathfrak{M}$. 

With an arbitrary subset $\bf E$ of the system of measures $\mathbf{M}$ we associate a class of functions that we call {\it admissible} for $\bf E$ and denote by $\Adm(\bf E)$. Namely,
\begin{eqnarray*}
\Adm({\bf E})
&=&
\Big\{f\colon X\to \mathbb R:\ f\ \text{is}\ m-\text{measurable},\ \ f\geq 0,\ \text{and}
\\  
&&\int_{X}f\,d\mu\geq 1,\ \text{for all}\  \mu\in {\bf E}\Big\}.
\end{eqnarray*}

\begin{definition}
 For $0<p<\infty$, the module $M_p({\bf E})$ of a system of measures ${\bf E}$ is defined as
\[
M_p({\bf E})=\inf\limits_{f\in \Adm({\bf E})}\int_{X}f^p\,dm,
\]
interpreted as $+\infty$ if $\Adm({\bf E})=\emptyset$.
\end{definition}

The reader can find the fundamental properties of the $p$-module of measures in~\cite[Chapter 1]{MR97720}.

A system of measures ${\bf E}\subset\mathbf M$ can be associated with the set where the measures are supported~\cite{MR2068961,MR0454009}. 

(I)\ Consider, for instance, {\it a family of rectifiable curves} $\Gamma=\{\gamma\colon[a_{\gamma},b_{\gamma}]\to \mathbb R^n\}$ and the associated  system of measures 
$$
{\mathbf E}=\{{\rm Var}\Big(\frac{d\gamma}{dt}\Big)=\mathcal H^1_{d_E}(|\gamma|):\ \ \gamma\in \Gamma\}.
$$
Here ${\rm Var}\Big(\frac{d\gamma}{dt}\Big)$ is the total variation of the vector valued Radon measure $\frac{d\gamma}{dt}$, which coincides with the Hausdorff measure $\mathcal H^1_{d_E}(|\gamma|)$ of the locus $|\gamma|$ of the curve $\gamma\in \Gamma$.

If we consider a subfamily $\widetilde \Gamma=\{\tilde\gamma\colon[\tilde a_{\gamma},\tilde b_{\gamma}]\to \mathbb R^n\}\subset \Gamma$ of absolutely continuous curves, then the corresponding measures can be calculated by 
\begin{equation*}\label{eq:acc}
{\rm Var}\Big(\frac{d\tilde\gamma}{dt}\Big)
=\int_{a_{\tilde \gamma}}^{b_{\tilde \gamma}}\Big|\frac{d\tilde \gamma(t)}{dt}\Big|\,dt
=\int_{|\tilde \gamma|}d\mathcal H^1_{d_E},\quad \tilde \gamma\in \widetilde \Gamma.
\end{equation*} 

(II)\ A {\it family of locally Lipschitz $k$-dimensional surfaces} in $\mathbb R^n$. Each surface locally is the image of an open set of $\mathbb R^k$ under a Lipschitz map $f$. In this case a surface measure locally coincides with $d\sigma=|J(f,t)|dt$. Here $J(f,t)$ is the Jacobian of $f$ for the points $t\in\mathbb R^k$, where the Jacobian $J(f,t)$ is defined. The corresponding family of measures is supported on these Lipschitz $k$-dimensional surfaces.

(III)\ A {\it family of countable $\mathcal H^k_{d_E}$-rectifiable subsets} in $\mathbb R^n$, where we understand the rectifiability in the sense of~\cite{MR0257325}. The $k$-dimensional Hausdorff measures $\mathcal H^k_{d_E}$, is the system of measures, associated with the family of countable $\mathcal H^k_{d_E}$-rectifiable sets.


\subsection{Exceptional families of measures}


A system $\mathbf E_0\subset \mathbf M$ is called $p$-exceptional, if $M_p(\mathbf E_0)=0$.
A statement concerning measures $\mu\in\mathbf M$ is said to hold $M_p$-almost everywhere if it fails to hold for a $p$-exceptional system $\mathbf E_0$. The question that we are interested in is to study $p$-exceptional sets of measures on  Carnot groups. Let us remind that a point-set $\mathbf E_0\subset X$ in a measure space $(X,m)$ has vanishing measure: $m(\mathbf E_0)=0$ if and only if there is a function $f\in L^p(X,m)$ such that $f(x)=+\infty$ for all $x\in \mathbf E_0$. A generalisation of this fact to a system of measures $\mathbf E_0\subset\mathbf M$ is given by B.~Fuglede.

\begin{theorem}\cite[Theorem 2]{MR97720}\label{th:excep}
A system of measures ${\bf E}_0\subset\mathbf M$ is $p$-excep-tional if and only if there exists a non-negative function $f\in L^p(X,m)$ such that 
$$
\int_Xf\,d\mu=+\infty\quad\text{for every}\quad\mu\in {\bf E}_0.
$$
\end{theorem}

\begin{remark}
{\rm By Theorem~\ref{th:excep}, it is easy to see that a family of curves in $\mathbb R^n$ that are not locally rectifiable is a $p$-exceptional set for $p\geq 1$. }
\end{remark}


\subsection{Exceptional family of curves on Carnot groups}

Recall the definition of horizontal subbundle $H\G\subset T\G$ from Section~\ref{sec:carnot-group}.
We say that a function $f\colon I\to \G$, $I\subset\mathbb R$ is $d_\rho$-Lipschitz continuous if it is Lipschitz continuous between metric spaces $(I,d_E)$ and $(\G,d_\rho)$.

\begin{definition}\label{horizontal} 
A continuous curve $\gamma\colon I\to \G$, $I\subset \mathbb R$, is called {\it horizontal} if it is $d_E$-Lipschitz continuous and the tangent vector $\dot\gamma(t)$ belongs to $H_{\gamma(t)}\G$ for almost all $t\in I$.
\end{definition}
We mention a well know example of a $p$-exceptional family of curves on a Carnot group. 
 
\begin{example}\label{not admissible}
We define ${\bf M}: =\{\mc H^1_{d_\rho} \res \gamma, \ \gamma\in\Gamma \}$,
where
$$
\Gamma:=\{\gamma\colon [0,1]\to \G \ \mbox{is 
a $d_E$-Lipschitz continuous curve}\}, 
$$
and
$$ 
\Gamma_H:=\{\gamma:[0,1]\to \G \ \mbox{is  a horizontal curve}\},
\quad
{\bf M}_H: =\{ \mc H^1_{d_\rho} \res \gamma, \ \gamma\in \Gamma_H\}.$$
We claim that
$
{\bf M} \setminus {\bf M}_H \quad\mbox{is $p$-exceptional for all $p>0$.}
$
By Theorem~\ref{th:excep} it is enough to find a nonnegative function $f\in L^p(\mathbb G,\mathbf g_{\G})$ such that 
$$
\int_{\gamma}f\,d\mathcal H^1_{d_\rho}=\infty
$$
for all $\gamma\in \Gamma \setminus \Gamma_H$.
Without loss of generality we can assume that the loci of curves $\gamma\in \Gamma$ are contained in a compact set $K\subset \mathbb G$.
Then the function
$$
f(x)=\begin{cases}
1,\quad&\text{if}\quad x\in K,
\\
0,\quad&\text{if}\quad x\notin K,
\end{cases}
$$
belongs to $L^p(\mathbb G,\mathcal L^N)$. 
Assume now, by contradiction, that there is 
$\gamma\in \Gamma \setminus \Gamma_H$ such that $\int_{\gamma}f\,d\mc H^1_{d_\rho}=\mc H^1_{d_\rho}(\gamma)<\infty$. 
We will show that $\gamma\in \Gamma_H$, yielding a contradiction. 
The proof is more or less standard, but we prefer to give complete arguments.
By the property of module of a minorized family of curves, see~\cite[Theorem 6.4]{MR0454009}  
we can assume that the curve $\gamma$ is injective. 
Since  $\gamma([0,1])$ is a closed and connected set, by~\cite[Theorem
4.4.8]{MR2039660} we can write 
\begin{equation}\label{rectifiability}
\gamma([0,1])=\gamma_0 \bigcup \Big(\cup_{k=1}^{\infty} \gamma_k([0,1])\Big).
\end{equation}
Here $\gamma_0$ is a Borel set such that
 $\mc H_{d_\rho}^1(\gamma_0)=0$ and $\gamma_k\colon[0,1]\to\gamma([0,1])$ are $d_\rho$-Lipschitz continuous functions. Note also that 
\begin{equation}\label{dc and eu}
d_E(x,y)\leq C d_{\rho}(x,y)\quad
\Longrightarrow\quad
0\le \mc H^1_{d_E}(\gamma) \le \mc H^1_{d_\rho}(\gamma)
\end{equation}
and that $\gamma^{-1}(\gamma_0)$, has zero Lebesgue measure  in $[0,1]$.
Indeed~\eqref{dc and eu} implies that $\mc H^1_{d_E}(\gamma_0)=0.$
Thus we can  apply the area formula of~\cite[Theorem 3.3.1]{MR2039660} for  $A_0:=\gamma^{-1} (\gamma_0) \subset[0,1]$: 
\begin{equation}\label{eq:11}\begin{split}
\mc L^1(A_0) 
= \int_{A_0} \Big|\frac{d\gamma (s)}{ds}\Big|\, d \mc L^1(s)
= \int_{\gamma_0} \mathrm{card}\,( \gamma^{-1}(y)) d \mc H^1_{d_E}(y) = 0.
\end{split}\end{equation}
Thus, by Rademacher's theorem, there exists $A\subset  [0,1]$ with 
$\mc L^1(A)=0$ such that for all $t\in A$   we have that $\gamma(t)\in \cup_k \gamma_k([0,1])$
and $\gamma$ is differentiable at $t$.
 
Since both $\mathbb R$ and $\G$ are Carnot groups, by Pansu-Rademacher
  theorem, all $\gamma_k$'s  are Pansu differentiable in a set $[0,1] \setminus A_1$
  with $\mc L^1(A_1)=0$. For any $k\in\mathbb N$, let us denote by $d_P\gamma_k(t)$ the Pansu differential at a point $t\in [0,1]\setminus A_1$.
 Set now $A:=A_0\cup A_1$.  Arguing as in the proof of \cite[Theorem 3.5 (2)]{FrSerSC07} the Euclidean tangent space
 to $\gamma$ at a point $\gamma (t)$, $t\in [0,1]\setminus A$, coincides with 
 $d_P\gamma_k(\tau)(\mathbb R)$ if $k,\tau$ are such that $\gamma(t)=\gamma_k(\tau)$. Since $d_P\gamma_k$
 is a group homomorphism between $\mathbb R$ and $\G$, it maps $\mathbb R$ into the first (horizontal) layer
 of $\G$, so that $\dot\gamma(t)\in H_{\gamma(t)}\G$, yielding a contradiction.
\end{example}

\begin{remark} By \cite[Theorem 4.2.1]{MR2039660} and \cite[Remark 4.1.3]{MR2039660},
we can always assume that
if $t\in [0,1]\setminus A$ then $|\dot\gamma(t)|=1$. Thus, if $\gamma(t)=\gamma_k(\tau)$,
and 
$$
\gamma_k(\tau) = \sum_{j=1}^{{\bf d}_1} u_j(\tau)X_{j1}(\gamma_k(\tau)),
$$
then $\|u \|_{L^\infty} \le C$,
where $C$ is independent of $\tau$. Thus, in the definition of $\Gamma_H$ we can replace
``horizontal'' by ``admissible'', see~\cite{MR3971262}.
\end{remark}

\subsection{Exceptional families of Radon measures on Carnot groups}\label{sec:Radon}


\begin{definition} If $\mu$ is a measure on a metric space $(X,\rho)$, and $h>0$, then the values
\begin{equation*}
\Theta^h_*(\mu,x) = \liminf\limits_{r\to 0} \frac{\mu(B_{\rho}(x,r))}{r^h},\quad \text{and} \quad
\Theta^{h,*}(\mu,x) = \limsup\limits_{r\to 0} \frac{\mu(B_{\rho}(x,r))}{r^h},
\end{equation*}
are called the lower and upper $h$-density of $\mu$ at the point $x\in X$, respectively. We say that measure $\mu$ has $h$-density $\Theta^h(\mu,x)$ if
$$
0<\Theta^h_*(\mu,x)=\Theta^h(\mu,x)=\Theta^{h,*}(\mu,x)<\infty.
$$
\end{definition}

\begin{lemma}\label{lem:L^p} 
Let a Carnot group $\G$ of topological dimension $N$ and homogeneous dimension $Q$ be endowed with a distance function $d_\rho$ of type $(D_2)-(D_4)$. Let $\mu$ be a Radon measure on $\G$. If 
$
\Theta^h(\mu,x)>0$ for $1<h<Q$ and $\mu$-{a.e.} $x\in \G$,
then
$\Theta^h(\mu,\cdot) \in L^p(\G, \mathbf g_{\G})$ for any $p>0$.
\end{lemma}
\begin{proof} Recall the relation $\mathbf g_{\G}\sim\mathcal L^N$.
We notice that the map $x\to \Theta^h(\mu,x)$ is Borel measurable,
see e.g.~\cite[Remark 3.1]{MR756417}. Let us show that
\begin{equation}\label{5.12}
\mathbf g_{\G}(\{x\in \G;\ \Theta^h(\mu,x)>0\})=0.
\end{equation}
Fix $R>0$ and assume by contradiction that
\begin{equation*}
\mathbf g_{\G}\big(B_{d_\rho}(e,R)\cap\{x\in \G\, ; \, \Theta^h(\mu,x)>0\}\big) >0.
\end{equation*}
We have
\begin{equation*}\begin{split}
\mathbf g_{\G}&\big(B_{d_\rho}(e,R)\cap\{x\in \G\, ; \, \Theta^h(\mu,x)>0\}\big)
\\&
=\lim_{k\to\infty} \mathbf g_{\G}\big(B_{d_\rho}(e,R)\cap\{x\in \G\, ; \, \Theta^d(\mu,x)>\frac1k\}\big).
\end{split}\end{equation*}
Denote $E_k=B_{d_\rho}(e,R)\cap\{x\in \G\, ; \, \Theta^h(\mu,x)>\frac1k\}$.
Then there exists $k\in\mathbb N$ such that
$
0< \mathbf g_{\G}\big(E_k\big) <\infty$.
Thus, by~\cite[Theorem 2.4.3]{MR2039660},
\begin{equation}\label{eq:mu-equivalence}
\mathcal H_{d_\rho}^h(E_k)  \le k\omega\, \mu(E_k) \le k\omega\mu\, (B_{d_\rho}(e,R)) < \infty,
\end{equation}
where $\omega$ is a normalisation constant.
On the other hand, the equivalence~\eqref{eq:QNequivalence}
implies that $\mathcal H_{d_\rho}^h(E_k)=\infty$ contradicting~\eqref{eq:mu-equivalence}. Letting $R\to\infty$
we obtain~\eqref{5.12}. This accomplishes the proof.
\end{proof}

\begin{example}\label{ex:Heisenberg-1}
 Consider the
Heisenberg group $H^1$ endowed with  a distance function $d_\rho$ of type $(D_2)-(D_4)$. 
Let $\mathbf M$ be the set of all Radon measures $\mu$
on $H^1$ satisfying
\begin{equation}\label{1 density}
\Theta^1(\mu,x):=\lim_{r\to 0} \, \dfrac{\mu(B_{d_\rho}(x,r))}{r} >0 \qquad\mbox{for  $\mu$-a.e. $x\in H^1$}.
\end{equation}
We let $\mathbf M_0\subset \mathbf M$ to be the measures for which there exists a countable family of Lipschitz maps $\Phi_i\colon A_i \to H^1$, $A_i \subset \mathbb R$,
such that 
$$\mu \big(H^1\setminus \bigcup_i \Phi_i (A_i)\big)=0.$$
We want to show that $ \mathbf M\setminus  \mathbf M_0$ is $p$-exceptional for $p>0$. 
By Theorem~\ref{th:excep} and Lemma~\ref{lem:L^p} it is enough
to show that
$$
\int_{H^1} \Theta^1(\mu,x)\, d\mu(x) = \infty\qquad\mbox{if $\mu\in \mathbf M\setminus  \mathbf M_0$.}
$$
Suppose by contradiction that the above integral is finite for a given measure $\mu\in\mathbf M$. Then $\Theta^1(\mu,x)<\infty$ for $\mu$-a.e. $x\in H^1$. Then applying~\cite[Theorem 1.4]{antonelli2021rectifiable}, which is an analog of one-dimensional Preiss' theorem for $H^1$, we obtain that $\mu\in \mathbf M_0$. That is a contradiction.

\end{example}

\begin{example}
In this example we refer to the definition of a tangent measure $\mathrm{Tan}_h(\mu,x)$ in~\cite[Chapter 14]{MR1333890},\cite{antonelli2021rectifiable}. Consider a Carnot group $\G$ as a metric space $(\mathbb G,d_\rho)$ where $d_\rho$ is one of the distance functions $(D_2)-(D_4)$. For $1 < h < Q$ denote by $\mathbf M$ the set of all Radon measures $\mu$
on $\mathbb G$ such that
\begin{itemize}
\item[i)] $\Theta_*^h(\mu, x)>0$ for $\mu$-a.e. $x\in \mathbb G$;
\item[ii)] $\mathrm{Tan}_h(\mu,x) \subset \{\lambda \mathcal S^h_{d_\rho}\res\mathbb V(x)\}$,
where $\mathbb V(x)$ is a {\it complemented} homogeneous subgroup in $\mathbb G$.
\end{itemize}
We let $\mathbf M_0\subset \mathbf M$ to be a family of Radon measures 
such that there exists a countable family $\{ \Gamma_i:= \mathrm{graph}\, (\phi_i)\, ; \, i=1,2,\dots\}$ of compact intrinsic Lipschitz
graphs of metric dimension $h$, see Section~\ref{sec:IntrinsicLipschitzsurface}, which are intrinsically
differentiable almost everywhere, see~\cite{MR3511465}, 
and such that
$$
\mu\Big( \mathbb G\setminus \bigcup_{i=1}^\infty\Gamma_i\Big) =0.
$$
Then we claim that $\mathbf M\setminus  \mathbf M_0$ is $p$-exceptional family of measures. By the methods of Lemma~\ref{lem:L^p} one can prove that $\Theta^{h,*}(\mu,\cdot) \in L^p(\mathbb G, \mathbf g_{\G})$ for any $p>0$. Thus it is enough 
to show that
$$
\int_{\G} \Theta^{h,*}(\mu,x)\, d\mu(x) = \infty\qquad\mbox{if $\mu\in \mathbf M\setminus  \mathbf M_0$}
$$
by Theorem~\ref{th:excep}.
Suppose by contradiction that the above integral is finite for a measure $\mu\in\mathbf M\setminus  \mathbf M_0$. Then $\Theta^{h,*}(\mu,x)<\infty$ for $\mu$-a.e. $x\in \mathbb G$ and therefore~\cite[Theorem 1.8]{antonelli2021rectifiable}
implies that
$\mu\in \mathbf M_0$, which is a contradiction.
\end{example}



\subsection{Families of surfaces on Carnot groups}


In this section we aim to define families of surfaces that will be of our interest.


\subsubsection{Intrinsic Lipschitz surfaces on the Carnot groups}\label{sec:IntrinsicLipschitzsurface}


In the present section we recall the definition of {\it intrinsic Lipschitz graphs}, i.e. graphs of {\it intrinsic Lipschitz functions}, see~\cite{MR2287539,MR3511465}. Then we give a definition of an intrinsic Lipschitz surface.
A subgroup $\M$ of a Carnot group $\G$ is called a homogeneous subgroup if $\M$ is a homogeneous group with respect to the dilation $\delta_{\lambda}$ defined in~\eqref{eq:delta}. Let us assume that $\G$ is decomposed into complementary homogeneous subgroups: $\G=\M\cdot \HH$, $\M\cap \HH=e$, and let $\mathbf P_{\M}$ and $\mathbf P_{\HH}$ be the canonical projections: $\mathbf P_{\M}\colon \G\to \M$ and $\mathbf P_{\HH}\colon \G\to \HH$ defined by the identity $\mathbf P_{\M}q\cdot \mathbf P_{\HH}q\equiv q$ for $q\in \G$.  The projections define intrinsic cones:
$$
C_{\M,\HH}(e,\beta)=\{p\in \G\mid\ \| \mathbf P_{\M} p\|\leq\beta \|\mathbf P_{\HH}p\|\},\qquad C_{\M,\HH}(q,\beta)=q\cdot C_{\M,\HH}(e,\beta),
$$
where $\beta>0$ is called the opening of the cone $C_{\M,\HH}(q,\beta)$ and $q$ is the vertex. 

\begin{definition}
The graph of a function $f\colon \Omega\to \mathbb H$, where $\Omega$ is an open set of $\M$, is the set 
$$
\graph f\ =\{\ q\cdot f(q)\in \G=\M\cdot \HH\mid\ q\in \Omega\subset\mathbb M\}.
$$
A function $f\colon \Omega\to \HH$, $\Omega\subset \M$, is an {\it intrinsic Lipschitz function} in $\Omega$ with the Lipschitz constant $L>0$ if
$$
C_{\M,\HH}(p,1/L)\cap \graph f =\{p\}\quad\text{for all}\quad p\in \graph f.
$$
An {\it intrinsic Lipschitz graph} is the graph of an intrinsic Lipschitz function. 
\end{definition}
Left translation
of intrinsic Lipschitz graphs are still intrinsic Lipschitz graphs. Following~\cite[Lemma 2.12]{MR3511465}, we set
\begin{equation}\label{eq:c0}
c_0(\mathbb M,\mathbb H):= \inf\{\norm{mh} \,:  \, \norm{m}+ \norm{h} =1 \}.
\end{equation}

\begin{remark}\label{rem:intrinsicgraph}
{\rm We emphasise that there is a subtle difference in the notions of a Lipschitz function between metric spaces and that of {\it intrinsic Lipschitz function} within a Carnot group. 
We refer the  reader to~\cite{MR3511465}.}
\end{remark}


\begin{definition}
The \emph{topological dimension} ${\bf d_t}$ of a (sub)group is the dimension of its Lie algebra. 
The \emph{metric dimension} ${\bf d_m}$ of a Borel set $U\subset \G$ is its Hausdorff dimension, with respect to the Hausdorff measure $\mathcal H_{d_\rho}$ (or $\mc S_{d_\rho}$) for a distance function $d_\rho$ of type $(D_2)-(D_4)$.  
We say that $\M$  is a \emph{$({\bf d_t},{\bf d_m})$-subgroup of $\G$} if $\M$ is a homogeneous subgroup of $\G$ with topological  dimension ${\bf d_t}$ and metric dimension ${\bf d_m}$. We say that $\graph f$ is intrinsic $({\bf d_t},{\bf d_m})$-Lipschitz grapf if $f\colon \Omega\to \HH$, $\Omega\subset \M$, is an intrinsic Lipschitz function and $\M$  is a $({\bf d_t},{\bf d_m})$-subgroup of $\G$.
\end{definition}

The metric dimension ${\bf d_m}$ of a homogeneous subgroup is an integer usually larger than its topological dimension ${\bf d_t}$, see~\cite{MR806700} and coincides with the homogeneous dimension, defined in Section~\ref{sec:carnot-group}.

\begin{definition}\label{def:Lipsurf} 
Suppose $1\le {\bf d_t}\le N-1$ and $1\le {\bf d_m}\le Q-1$. A non-empty subset $S\subset\mathbb{G}$ is called an intrinsic $({\bf d_t},{\bf d_m})$-Lipschitz 
surface (or manifold in the graph representation) in $\G$ if to every point $x\in S$ there correspond an open neighbourhood $U(x,r)\subset\G$,
a decomposition $\G=\mathbb M_U\cdot \mathbb H_U$ and an open set $\Omega\subset\M_U$ such that 
 
\begin{itemize}
\item $x\in U$;
\item $\M_U$  is a $({\bf d_t},{\bf d_m})$-subgroup of $\G$;
\item there exists an intrinsic Lipschitz map $f_U\colon\Omega\to\mathbb H_U$ such that $S\cap U= \graph{f_U}\cap U(x,r)$.
\end{itemize}
\end{definition}


\subsubsection{Measures on the intrinsic Lipschitz graphs and surfaces}\label{sec:surface-measures}

We suppose that $(\mathbb G, \mathcal B, \mathcal L^N,d_\rho)$ is a Carnot group with the Borel $\sigma$-algebra $\mathcal B$, the Lebesgue measure $\mathcal L^N$ which is identified with the Haar measure $\mathbf g_{\G}$, and a distance function $d_\rho$ of types $(D_2-D_4)$. We assume that $\G=\mathbb M\cdot \mathbb H$. 
The following result provides the construction of a Borel measure on an intrinsic Lipschitz graph.

\begin{theorem}\cite[Theorem 3.9]{MR3511465}\label{th:FrSer}
Let $S$ be an intrinsic $({\bf d_t},{\bf d_m})$-Lipschitz 
graph on a Carnot
group $(\G,d_\rho)$. 
Suppose $S=\graph f$ is defined by a decomposition  $\G= \mathbb M\cdot \mathbb H$
and  an intrinsic $L$-Lipschitz function $f\colon\Omega \to\mathbb H$ in the domain $\Omega\subset\M$. Then 
there are positive constants $c_0,c$ depending on the decomposition $\G=\mathbb M\cdot \mathbb H$ such that
\begin{equation}\label{eq:Alhfors_reg}
\Big(\dfrac{c_0(\mathbb M\cdot \mathbb H) }{1+L}\Big)^{{\bf d_m}}
R^{{\bf d_m}}\leq\mc S^{{\bf d_m}}_{d_\rho}(S\cap B_{d_\rho}(x,R))\leq c(\mathbb M\cdot \mathbb H)(1+L)^{d_{m}}R^{{\bf d_m}}
\end{equation}
for all points $x\in S$ and $R>0$.
In particular, the Hausdorff dimension of $S$ with respect to $d_\rho$ equals the homogeneous dimension of the group $\M$.
\end{theorem}
%
Let $\G=\M\cdot \HH$ be a decomposition of $\G$ into complementary homogeneous subgroups.
If $\Omega\subset \M$ is an open set and $f\colon \Omega\to\HH$ is an intrinsic Lipschitz function, then one can define a map $\Phi_f\colon\Omega\to \mathbb G$ by $\Phi_f(m)=m\cdot f(m)$, $m\in\Omega$, that parametrises the intrinsic Lipschitz graph of $f$.
We define two measures
\begin{equation}\label{eq:measure-sigma}
\sigma_S(A)=\mathcal S^{{\bf d_m}}_d\res{\graph{f}}(A)=\mathcal S^{{\bf d_m}}_d(\graph{f}\cap A)
\end{equation}
and
\begin{equation}\label{eq:measure-mu}
\mu(A)=\Big(\big((\Phi_f)_{\sharp}\big)\mathbf g_{\M}\Big)(A)=\mathbf g_{\M}(\Phi_f^{-1}(A))=\mathbf g_{\M}(\Phi_f^{-1}( \graph{f}\cap A))
\end{equation}
for any Borel measurable subset $A\subset \G$. 
Both measures $\sigma_S$ and $\mu$ are concentrated on the set $\graph{f}\subset\G$.

\begin{theorem}
For the measures $\sigma_S$ and $\mu$ defined in~\eqref{eq:measure-sigma} and~\eqref{eq:measure-mu} there are positive constants $C_1$ and $C_2$ such that 
\begin{equation}\label{eq:nu-mu}
C_1\sigma_S(A)\leq\mu(A)\leq C_2\sigma_S(A)
\end{equation}
for any Borel measurable set $A\subset \G$.
\end{theorem}
\begin{proof}
Notice that if $A\subset \G$, then by definition
\begin{equation}\label{eq:Phi=P}
\mu(A)=\mathbf g_{\M}(\Phi_f^{-1}(\graph{f}\cap A))=\mathbf g_{\M}(\mathbf P_{\M}(\graph{f}\cap A)).
\end{equation}
Let $B_{d_\rho}(x,R)\subset \G$ be a ball centred at $x\in \graph{f}$. Then,
by ~\cite[Formula (44)]{MR3511465},
$$ 
\mathbf P_{\M}(B_{d_\rho}(x,cR))\subset \mathbf P_{\M}(\graph{f}\cap B_{d_\rho}(x,R))\subset \mathbf P_{\M}(B_{d_\rho}(x,R)),
$$
where $c=\dfrac{c_0(\M\cdot\HH) }{1+L}$. Moreover, it was shown in~\cite[Lemma 2.20]{MR3511465} that 
$$
\mathbf g_{\M}\Big(\mathbf P_{\M}\big(B_{d_\rho}(x,R)\big)\Big)=c_1R^{{\bf d_m}}, \quad c_1=\mathbf g_{\M}(B_{d_\rho}(e,1)).
$$
It implies
\begin{equation}\label{May1 eq:1}
c_1c^{{\bf d_m}}R^{{\bf d_m}}\leq \mathbf g_{\M}\Big(\mathbf P_{\M}\big(\graph{f}\cap B_{d_\rho}(x,R)\big)\Big)=\mu(B_{d_\rho}(x,R))\leq c_1R^{{\bf d_m}}
\end{equation}
by~\eqref{eq:Phi=P} and the definition of the measure $\mu$. Passing to the upper and lower limits in~\eqref{May1 eq:1} we obtain 
\begin{equation}\label{densities}\begin{split}
c_1c^{{\bf d_m}} & \leq\liminf_{R\to 0}\frac{\mu(B_{d_\rho}(x,R))}{R^{{\bf d_m}}}
\leq
\limsup_{R\to 0}\frac{\mu(B_{d_\rho}(x,R))}{R^{{\bf d_m}}}\leq  c_1.
\end{split}\end{equation}
Therefore, arguing as in~\cite[Section 2.10.19]{MR0257325}, we can write
\eqref{May1 eq:1} as
\begin{equation}\label{densities}\begin{split}
\tilde c_1c^{{\bf d_m}}  \leq\Theta^{{\bf d_m}}_{*}(\mu,x)
\leq
\Theta^{{\bf d_m},*}(\mu,x)\leq \tilde c_1
\end{split}\end{equation}
for any $x\in \graph{f}$.
Again by~\cite[Section 2.10.19]{MR0257325} if follows  that 
\begin{equation}\label{eq:comparison}
\tilde c_1c^{{\bf d_m}}\mathcal  S^{{\bf d_m}}_{d_\rho}(U)\leq \mu(U)\leq \tilde c_12^{{\bf d_m}}\mathcal S^{{\bf d_m}}_{d_\rho}(U)
\end{equation}
for any Borel measurable set $U\subset\G$. As a set $U$ we take 
$U:= A\cap \graph{f}$ for a Borel set $A\subset \G$. By~\eqref{eq:comparison} we get
$$
C_1\sigma_S(A)=C_1\mathcal S^{{\bf d_m}}_{d_\rho}(U)\leq \mu(U) = \mu(A) \leq C_2 \mathcal S^{{\bf d_m}}_{d_\rho}(U)=C_2\sigma_S(A),
$$
where $C_1:= \tilde c_1c^{{\bf d_m}}$ and $C_2:= \tilde c_12^{{\bf d_m}}$.
\end{proof}

\begin{corollary}\label{cor:enequality}
Let $S=\graph{f}$ be an intrinsic $({\bf d_t},{\bf d_m})$-Lipschitz 
graph in a Carnot
group $(\G,d_\rho)$,
and let $\Phi_f\colon\Omega\to \mathbb G$ 
be the map defined by $\Phi_f(m)=m\cdot f(m)$, that parametrizes $S$.
For the measure $\sigma_S$ 
defined above and any Borel non-negative function $h$ on $\G$ one has
\begin{equation}\label{eq:integral_sigma}
C_1\int_{\G}h(y)\,d\sigma_S(y)\leq\int_{\Omega}(h\circ\Phi_f)(x)\,d\mathbf g_{\M}(x)\leq C_2\int_{\G}h(y)\,d\sigma_S(y).
\end{equation}
\end{corollary}

\begin{proof} 
We recall a result from~\cite[Theorem 1.19]{MR1333890}.
Let $X$ and $Y$ be two separable metric spaces, $\Phi\colon X\to Y$ a Borel map, $\nu$ a Borel measure on $X$ and $h$ is a Borel non-negative function on $Y$. Then
$$
\int_Yh(y)\,d(\Phi_{\sharp}\nu)(y)=\int_X(h\circ\Phi)(x)\,d\nu(x).
$$
We apply the result for the surface locally parametrised by an intrinsic  Lipschitz graph of $f$, taking $X=\Omega\subset\M$, $Y=\G$, $\Phi=\Phi_f$, $\nu=\mathbf g_{\M}$. Then~\eqref{eq:integral_sigma} follows.
\end{proof}

\begin{remark}\label{C1C2} 
We stress that, if $S$ is an intrinsic $({\bf d_t},{\bf d_m})$-Lipschitz 
surface, then according to Definition~\ref{def:Lipsurf}, the constant
$c_0(\mathbb M,\mathbb H)$ in~\eqref{eq:c0} depends on $S$ and therefore we can write
$c_0(\mathbb M,\mathbb H)=: c_0(S)$ when we consider this constant on a surface $S$.

If in Corollary~\ref{cor:enequality} we track the definition of the constants $C_1,C_2$, then
we see that they depend (up to geometric constants) only on the decomposition $\mathbb M\cdot\mathbb H$ and the Lipschitz
constant $L$. So, we can say that $C_1,C_2$ depend on $S$ and we write $C_1(S)$,
$C_2(S)$.
\end{remark}


\subsubsection{Examples of families of surfaces on Carnot groups}\label{subsec:Examples}


\begin{example} {\it $(1,1)$-intrinsic Lipschitz surfaces.}
Let $\M_X=\exp (tX)$, $t\in \mathbb R$, $X\in \mathfrak g_1$ be a one dimensional commutative subgroup and $\HH_X$ a complementary to $\M_X$ subgroup. Let $\phi_X\colon \M_X\to \HH_X$ be a Lipschitz map. The family
$$
\Phi=\{\phi_X\colon \M_X\to \HH_X:\ \ X\in \mathfrak g_1\}
$$  
is a family of $(1,1)$-intrinsic Lipschitz surfaces, that is a family of horizontal curves.
\end{example}

\begin{example} {\it A family of parametrised intrinsic Lipschitz graphs.}
Let $\M$ and $\HH$ be complementary subgroups and $f\colon \Omega\to\HH$, $\Omega\subset \M$, be an intrinsic Lipschitz function with $S=\graph{f}$. We define 
$$
f_\lambda\colon \delta_\lambda\Omega\to\HH:\ 
f_{\lambda}(m):=\delta_{\lambda}f(\delta_{1/\lambda}m).
$$
Then $S_{\lambda}=\graph{f_{\lambda}}$
is a family of intrinsic Lipschitz graphs, parametrised by $\lambda>0$ and it coincides with $\delta_{\lambda}S$.  

We also consider 
$$
f_q\colon\Omega_q\to\HH, \quad \Omega_q=\{m\in\M:\ \mathbf P_{\M}(q^{-1}m)\in\Omega\},\quad q\in E\subseteq\G
$$
defined by
$$
f_q(m)=\big(\mathbf P_{\HH}(q^{-1}m)\big)^{-1}\cdot f(\mathbf P_{\M}(q^{-1}m)).
$$
Then 
$$
S_q=\graph{f_q}=\{\big(m\cdot f_q(m)\big):\ m\in\Omega_q,\ \  q\in E\subset\G,\ \ f_q\colon \Omega_q\to\HH\}
$$ 
is a family of intrinsic Lipschitz graphs parametrised by $q\in E\subset\G$ and it coincides with $q\cdot S=L_q(S)$. 
The details about the properties of $S_{\lambda}$ and $S_q$ see in~\cite[Theorem 3.2]{MR3511465}.

Particularly, if $q\in E=\mathbb H$, then $\mathbf P_{\M}(q^{-1}m)=m$, $\mathbf P_{\HH}(q^{-1}m)=q^{-1}$. It implies $\Omega_q=\Omega$, and the family $f_q(m)=q\cdot f(m)$ is a family of graphs shifted along the subgroup $\HH$.
\end{example}

\begin{example}
Let $\mathcal F\in\Aut(\G)$ be a  grading preserving automorphism. Let $f\colon \M\to \HH$ be an (intrinsic) Lipschitz function with $S=\graph{f}$.  We define 
$
f_{\mathcal F}\colon\mathcal F(\M)\to\mathcal F(\HH)
$
by
$
f_{\mathcal F}(m)=\big({\mathcal F} f{\mathcal F}^{-1}\big)(m).
$
Then 
$$
S_{\mathcal F}=\graph{f_{\mathcal F}}=\{\big(m\cdot f_{\mathcal F}(m)\big):\ m\in\mathcal F(\M)\}
$$ 
is a family of (intrinsic) Lipschitz graphs and it coincides with $\mathcal F(S)$. Indeed, let $(m\cdot f(m))\in S$, then 
\begin{eqnarray*}
\mathcal F (S)\ni\mathcal F((m\cdot f(m))&=&(\mathcal F m\cdot \mathcal F f(m))=
(\mathcal F m\cdot \mathcal F f(\mathcal F^{-1}\mathcal Fm))
\\
&=&
(m'\cdot \mathcal F f (\mathcal F^{-1}m'))=
(m'\cdot f_{\mathcal F}(m'))\in S_{\mathcal F}.
\end{eqnarray*}

If $\mathcal F\in \Aut(\G)$ preserves the homogeneous norm, then an intrinsic $L$-Lipschitz graph is transformed to an intrinsic $L$-Lipschitz graph. 
\end{example}


\subsection{Exceptional families of intrinsic Lipschitz surfaces}\label{sec:ExceptionalLipSurf}



\subsubsection{Exceptional families for $0<p<1$}


Denote by $\Sigma^{({\bf d_t},{\bf d_m})}$ a family of intrinsic $({\bf d_t},{\bf d_m})$-Lipschitz surfaces in $\mathbb G$. With each surface $S\in \Sigma^{({\bf d_t},{\bf d_m})}$ we associate a measure $\sigma_S$ as in Section~\ref{sec:surface-measures}. 
Let $\Sigma\subset \Sigma^{({\bf d_t},{\bf d_m})}$ be a subfamily and  ${\bf E}$ the system of measures $\sigma_S$, $S\in \Sigma$, associated with $\Sigma$. Then $M_p(\Sigma)=M_p({\bf E})$ denotes the $p$-module of the family of measures $\mathbf E$ as well as the family of the surfaces $\Sigma$.

Note that in~\cite[page 187]{MR97720} it was shown that if $0<p<1$, then any system of Lipschitz $k$-dimensional surfaces $\Sigma$ which intersects the cube
$$
{\rm Cube}_a=\{x\in \mathbb R^n\mid\ |x_{l}|<a,\ l=1,\ldots,n\}
$$
has vanishing $p$-module for any $a>0$. It leads to the fact that $M_p(\Sigma)=0$, $p\in(0,1)$, by the monotonicity of $p$-module. We will study an analog situation on the Carnot groups.

Suppose $\G=\mathbb M\cdot \mathbb H$ is a decomposition of $\G$, and denote by $\mathfrak m$ and $\mathfrak g$ the Lie algebras of the Lie groups $\mathbb M$ and $\G$, respectively. Let us fix a weak Malcev basis $\{W_1^{\mathbb M},\dots,W_{{\bf d_t}}^{\mathbb M},W_{{\bf d_t}+1},\dots,W_{N}\}$ for $\mathfrak g$ through $\mathfrak m$, see~\cite[Theorem 1.1.13]{corwin1990representations}. 
We define the following coordinate maps 
$$
T_{\mathbb M}\colon\mathbb R^{{\bf d_t}}\to \mathbb M:\quad \zeta=(\zeta_1,\ldots,\zeta_{{\bf d_t}})\mapsto \exp(\zeta_1W_1^{\mathbb M})\cdot\ldots\cdot\exp(\zeta_{{\bf d_t}}W_{{\bf d_t}}^{\mathbb M}),
$$
$$
T_{\mathbb G}\colon\mathbb R^{N-{\bf d_t}}\to \mathbb G:\quad s=(s_{{\bf d_t}+1},\ldots,s_N)\mapsto \exp(s_{{\bf d_t}+1}W_{{\bf d_t}+1})\cdot\ldots\cdot\exp(s_NW_N),
$$
$$
T\colon\mathbb R^{N-{\bf d_t}}\to \mathbb M\setminus\mathbb G:\ s=(s_{{\bf d_t}+1},\ldots,s_N)\mapsto \mathbb M\cdot\exp(s_{{\bf d_t}+1}W_{{\bf d_t}+1})\cdot\ldots\cdot\exp(s_{N}W_{N}).
$$
We define the natural projection on the right coset space by
\begin{equation}\label{projection pi}
\pi\colon\G\to \mathbb M\setminus \G:\quad g\mapsto \mathbb M\cdot g.
\end{equation}
The group $\G$ acts on the right on the coset space $\mathbb M\setminus \G$ by
$$
\big(\mathbb M\setminus \G\big)\times\mathbb G\to\mathbb M\setminus \G:\quad(\mathbb M\cdot g,\tilde g)\mapsto \mathbb M\cdot g\tilde g.
$$
Since $\G=\mathbb M\cdot \mathbb H$, the coset space is ``parametrized'' by the elements of $\mathbb H$ in the following sence
$
\mathbb M\cdot g=\mathbb M\cdot mh=\mathbb M\cdot h
$. Moreover $
\mathbb M\cdot g\tilde g=\mathbb M\cdot mh\tilde m\tilde h=\mathbb M\cdot\tilde{\tilde h}
$,
where $\tilde{\tilde h}$ is not necessarily $h\tilde h$.
In the following proposition we formulate a Fubini type theorem related to the right quotient space $\mathbb M\setminus \G$.

\begin{proposition}\cite[Lemma 1.2.13]{corwin1990representations}~\cite[Theorem 15.24]{ross1963abstract}
The map $T=\pi\circ T_{\mathbb G}$ is a diffeomorphism and the push forward measure $\mathbf g_{\mathbb M\setminus\mathbb G}=T_{\#}(\mathcal L^{N-{\bf d_t}})$ is a right $\mathbb G$-invariant measure on the coset space $\mathbb M\setminus\mathbb G$,~\cite[Theorem 1.2.12]{corwin1990representations}. Moreover the right invariant measure $\mathbf g_{\G}$ on $\mathbb G$ and the  measure $\mathbf g_{\mathbb M\setminus\mathbb G}$ on $\mathbb M\setminus\mathbb G$ are related by 
\begin{equation}\label{eq: measure relation}
\int_{\mathbb G}\varkappa(g)d\mathbf g_{\G}=\int_{\mathbb M\setminus\mathbb G}d\mathbf g_{\mathbb M\setminus\mathbb G}\int_{\mathbb M}\varkappa(mg)d\mathbf g_{\mathbb M}
\end{equation}
for any continuous function $\varkappa$ with a compact support.
\end{proposition}

By making use of the Vitali covering lemma, we consider a countable family of Euclidean balls $\{B(\xi_j, r)\subset \mathbb R^{{\bf d_t}}, j\in\mathbb N\}$ such that
\begin{equation}\label{eq:Vitali}
\begin{array}{lll}
&\bullet&\  \{B(\xi_j, r),\ j\in\mathbb N\}\  \text {is an open covering of}\  \mathbb R^{{\bf d_t}};
\\
&\bullet&\  \text{the balls}\  B(\xi_j,r/5)\  \text{are disjoint}.
\end{array}
\end{equation}
In particular, if a finite family of balls $\mathfrak B=\{B(\xi_{j_1},3r),\dots,B(\xi_{j_K},3r)\}$ 
has nonempty intersection, then $\#\mathfrak  B\le 30^{{\bf d_t}}$.
Indeed, suppose for sake of simplicity that $\mathfrak  B=\{B(\xi_1,3r),\dots,B(\xi_K,3r)\}$
are such that $\zeta_0\in \cap_{i=1}^K B(\xi_i,3r)$. By triangle inequality
$\cup_{i=1}^K B(\xi_i,3r)\subset B(\xi_1,6r)$. Since $\{B(\xi_1,r/5),\dots,B(\xi_K,r/5)\}$
is a disjoint family in $B(\xi_1,6r)$ we obtain
\begin{equation*}\begin{split}
{\rm vol(B(0,1))}\frac{r^{\bf d_t}}{5^{{\bf d_t}} }\cdot(\#\mathfrak  B) &= | \cup_{i=1}^K B(\xi_i,r/5)| 
\le |B(\xi_1,6r)| = (6r)^{{\bf d_t}}{\rm vol(B(0,1))}.
\end{split}\end{equation*}
Let now $\psi\in C^{\infty}_0(B(0,3r))$ be a cut-off function of the ball $B(0,2r)$
and set
\begin{equation}\label{eq:phi_r}
\phi(\zeta):= \phi_r(\zeta) = \sum_i 2^{-i} |\xi_j - \zeta|^{-1}
\psi(\xi_j - \zeta),
\end{equation}
where $\xi_j$ are centers of the balls in family~\eqref{eq:Vitali}.
\begin{lemma} For any $r>0$ and $0<1<p$ we have $\phi_r\in L^p(\mathbb R^{{\bf d_t}})$.
\end{lemma}
\begin{proof} If $\zeta$ is fixed, then $\psi(\xi_i- \zeta)>0$ if and only if $\zeta \in B(\xi_i,3r)$,
which is possible for at most $30^{{\bf d_t}}$ values of $j$. Thus, if $0<p<1$, then
$$
|\phi_r(\zeta)|^p \leq \big( \sum_i 2^{-i}  |\xi_i - \zeta|^{-1}\psi(\xi_i - \zeta)\big)^p \le C_p \sum_i 2^{-ip}  |\xi_i - \zeta|^{-p}\psi^p(\xi_i - \zeta),
$$
so that
\begin{equation*}\begin{split}
&\int_{\mathbb R^{{\bf d_t}}}  \phi_r(\zeta)^p \le C_p \sum_i 2^{-ip}
\int_{\mathbb R^{{\bf d_t}}}  |\xi_i - \zeta|^{-p}\psi^p(\xi_i - \zeta) \,d\zeta
\\& 
\le
C_p \sum_i 2^{-ip}
\int_{B(\xi_i,3r)}  |\xi_i - \zeta|^{-p} \,d\zeta
=  C_p \sum_i 2^{-ip}
\int_{B(0,3r)}  |\zeta|^{-p} \,d\zeta
<\infty.
\end{split}\end{equation*}
\end{proof}

\begin{lemma}\label{lem:infty}
 Let us fix the Vitali covering as in~\eqref{eq:Vitali} by balls of a radius $r$. Then for any $x\in B(\xi_i,r)$ the function $\phi_{r}$ defined in~\eqref{eq:phi_r} satisfies
\begin{equation}\label{eq:phi_r-infty}
\int_{B(x,r)}\phi_{r}(\zeta)\,d\zeta=\infty.
\end{equation}
\begin{proof} We have 
$
B(x,r)\subset B(\xi_i,2r)
$
by the triangle inequality and $x\in B(\xi_i,r)$. Then for $\zeta\in B(x,r)$ we get
$$
\phi_{r}(\zeta)\geq 2^{-i}|\xi_i - \zeta|^{-1}\psi(\xi_i - \zeta)=2^{-i}|\xi_i - \zeta|^{-1}. 
$$
It implies
\begin{eqnarray*}
\int_{B(x,r)}\phi_{r}(\zeta)\,d\zeta\geq
\int_{B(x,r)} 2^{-i} |\xi_j - \zeta|^{-1}\,d\zeta
\geq 
\int_{B(\xi_i,\tilde r)} 2^{-i} |\xi_j - \zeta|^{-1}\,d\zeta
=\infty,
\end{eqnarray*}
for some $\tilde r\in(0,r)$.
\end{proof}
\end{lemma}

A Carnot group $\G$ can admit various decompositions $\G=\M\cdot\mathbb H$ into homogeneous subgroups of the same topological ${\bf d_t}$ and Hausdorff ${\bf d_m}$ dimensions. The general result about the structure of such kind of decompositions and their number is not known. Therefore we restrict to the following cases. The first one when there are finitely many decompositions $\G_{\alpha}=\mathbb M_{\alpha}\cdot \mathbb H_{\alpha}$, $\alpha=1,2,\ldots, l$ into non isomorphic pairs $\mathbb M_{\alpha}\cdot \mathbb H_{\alpha}$. The second case when each pair $\mathbb M_{\alpha}\cdot \mathbb H_{\alpha}$ belongs to an orbit of the action of grading preserving isometries of~$\G$. 

It is enough to consider surfaces belonging to an open bounded set $\mathcal U\subset\G$, for instance a ball. Let $\Sigma$ be a system of Lipschitz surfaces with some specific property, that will be specified in theorems below, and let
$$
\mathbf E:=\{ \sigma_S=\mc S^{{\bf d_m}}_{d_\rho}\res S\, ; \, S\in \Sigma\}
$$
be the system of the associated measures in $(\G,d_\rho)$. Then we denote by
$\Sigma_{\mc U}$ the system of the Lipschitz surfaces $S\cap\mc U$, $S\in \Sigma$, and by ${\bf E}_{\mc U}$  the family of associated measures. If we show that $M_p(\mathbf E_{\mc U})=0$, then it will imply that the system of measures $\mathbf E$ is exceptional by~\cite[Theorem 3 (b)]{MR97720}.

We start from the family of a most simple nature, that is a family of graphs parametrised over a single decomposition $\G=\M\cdot\mathbb H$. Then we consider multiple decompositions and more complicate families of surfaces.

\begin{theorem}\label{th:1decomposition}
Let $\G=\mathbb M\cdot \mathbb H$ be a decomposition of $\G$, and let 
$\Sigma$ be a family of intrinsic $({\bf d_t},{\bf d_m})$-Lipschitz graphs over $\mathbb M$ and $\Sigma_{\mc U}$ the family of  $({\bf d_t},{\bf d_m})$-Lipschitz graphs in a bounded open set $\mc U\subset\G$.
Then the system $\mathbf E_{\mc U}$ is $p$-exceptional for $p\in (0,1)$.
\end{theorem}

\begin{proof} 
By definition, for any $S\in\Sigma_{\mc U}$ there exists an intrinsic Lipschitz function $f_S
\colon\Omega_S\to\mathbb H$, , such that 
\begin{itemize}
\item $\Omega_S$ is an open subset of $\mathbb M$;
\item $S=\graph{ f_S}$, i.e. $S=\Phi_S (\Omega_S)$,
where $\Phi_S(m)=m\cdot f_S(m)$, $m\in \Omega_S$;
\end{itemize}

If $R,N>0$ we denote by $\Sigma_{\mc U}(R,N)\subset  \Sigma_{\mc U}$ the family of graphs such that 
\begin{itemize}
\item the open set $T_{\mathbb M}^{-1} \Phi_{S}^{-1}(\mc U)$ contains an Euclidean
ball $B(\zeta_S, R)\subset \mathbb R^{{\bf d_t}}$;
\item the associated
measures $\sigma_S=\mc S^{{\bf d_m}}_{d_\rho}\res S$ satisfy
$$
\int_{\Phi_S^{-1}(\mc U)}(h\circ\Phi_S)(x)\,d\mathbf g_{\M}(x)\leq N \int_{S\cap \mc U}h(y)\,d\sigma_S (y).
$$
\end{itemize}

We fix $R$ and $N$ and denote ${\bf E}_{\mc U}(R,N)$ the family of associated measures to $\Sigma_{\mc U}(R,N)$. To show that $M_p\big({\bf E}_{\mc U}(R,N)\big)=0$ 
it is enough to find $F\in L^p(\G,\mathbf g_{\mathbb G})$ such that 
$
\int_\G F\, d\sigma_S = \infty$ for all $\sigma_S\in \mathbf E_{\mc U}(R,N)$, see Theorem~\ref{th:excep}.
We set
\begin{equation}\label{May1 eq:3}
F= \phi_R \circ T^{-1}_{\mathbb M}  \circ \Pi_{\mathbb M}\colon \mathbb G\to [0,\infty],
\end{equation}
where $\Pi_{\mathbb M}$ is the projection over $\mathbb M$ associated
with $\G=\mathbb M\cdot \mathbb H$, as in~\cite[Formula (28)]{MR3511465},
and $\phi_R$ is the function defined in~\eqref{eq:phi_r} for $r=R$.

\noindent {\bf Step 1: we claim that  $F\in L^p(\G,\mathbf g_{\mathbb G})$.} Since $\mathcal U\subset K\subset\G$ for some compact set $K$, it is enough to show that 
$F_K=\chi_K F\in L^p(\G,\mathbf g_{\mathbb G})$.
If we put $\tilde K:= \pi(K)$, then $\tilde K\subset \mathbb M/\mathbb G$ is compact by the 
continuity of $\pi$ and $ \chi_K(g)\le \chi_{\tilde K} \big( \pi(g)\big)$, since,
if $g\in K$, then $\pi(g)\in \tilde K$.

We apply~\eqref{eq: measure relation} with $\varkappa=F^p_K$ and obtain
\begin{equation*}\begin{split}
&\int_\G |F_K|^p(g)\,  d\mathbf g_{\G}(g) 
= \int_{\mathbb M\setminus\G}\, d\mathbf g_{\mathbb M\setminus \G}
\int_{\mathbb M} |F_K|^p(mg)\, d\mathbf g_{\mathbb M}(m)
\\
&=
\int_{\mathbb M\setminus\G}\, d\mathbf g_{\mathbb M\setminus \G}
\int_{\mathbb M} \chi_K (mg)
|\phi_R \circ T^{-1}_{\mathbb M}\circ\Pi_{\mathbb M}(m)|^p\, d\mathbf g_{\mathbb M}(m)
\\&
\le
\int_{\mathbb M\setminus\G}\, \chi_{\tilde K} (\pi(g))d\mathbf g_{\mathbb M\setminus \G}
\int_{\mathbb M}
|\phi_R \circ T^{-1}_{\mathbb M}(m)|^p\, d\mathbf g_{\mathbb M}(m)
\\&
= c(\tilde K) \int_{\mathbb M}
|\phi_R \circ T^{-1}_{\mathbb M}(m)|^p\, d\mathbf g_{\mathbb M}(m)
= c(\tilde K) \int_{\mathbb R^{{\bf d_t}}}
|\phi_R|^p (\zeta)\, d\mc L^{{\bf d_t}}(\zeta) < \infty.
\end{split}\end{equation*}
\medskip

\noindent {\bf Step 2: we claim that $\int_S F\,d\sigma_S = \infty$ for any $\sigma_S\in \mathbf{E}_{\mc U}(R,N)$.} Assume that
$S\in \Sigma_{\mc U}(R,N)$ is the graph of a Lipschitz function $f_S\colon\Omega_S\to \mathbb H$.
We denote by
$\Phi_S(m):=m\cdot f_S(m)$, $m\in \Omega_S$, the parametrization of $S$ associated
with $f_S$. We stress that $\big(\Pi_{\mathbb M}\circ \Phi_S\big) (m)=m$. By
Corollary~\ref{cor:enequality} with $N=C_2$, and~\eqref{May1 eq:3} we have:
\begin{equation*}\begin{split}
&N\int_{S\cap \mc U} F \, d\sigma_S
\ge \int_{\Phi_S^{-1}(\mc U)} F\circ \Phi_S (m)\, d\mathbf g_{\mathbb M} 
=  \int_{\Phi_S^{-1}(\mc U)} \phi_R \circ T^{-1}_{\mathbb M}\circ \Pi_{\mathbb M}\circ \Phi_S(m)\, d\mathbf g_{\mathbb M} 
\\&
=  \int_{\Phi_S^{-1}(\mc U)} \phi_R \circ T^{-1}_{\mathbb M}(m)\, d\mathbf g_{\mathbb M} 
=  \int_{T_{\mathbb M}^{-1} \Phi_{S}^{-1}(\mc U)} \phi_R (\zeta)\, d\zeta
\ge 
 \int_{B(\zeta_S, R)} \phi_R (\zeta)\, d\zeta.
\end{split}\end{equation*}
By the Vitali covering lemma, there exists $\xi_i$ such that $\zeta_S\in B(\xi_i,R)$. Thus we apply Lemma~\ref{lem:infty} to show that the integral on the right-hand diverges. 

\noindent {\bf Step 3: we claim that $M_p(\Sigma_{\mathcal U})=0$.} From Steps 1 and 2 we conclude that $M_p(\Sigma_{\mathcal U}(R,N))=0$ for any $R,N>0$. We set $R=\frac{1}{M}$ for $M\in\mathbb N$. Then
$$
\Sigma_{\mc U} = \bigcup_{N\in\mathbb N}\bigcup_{M\in\mathbb N} \Sigma_{\mc U}(\frac{1}{M},N).
$$
We conclude by~\cite[Theorem 3 (b)]{MR97720} that $M_p(\Sigma_{\mathcal U})=0$.
\end{proof}

\begin{corollary}
Let $\G=\mathbb M\cdot \mathbb H$ be a decomposition of $\G$ and let 
$\Sigma$ be a family of intrinsic $({\bf d_t},{\bf d_m})$-Lipschitz surfaces, such that locally each surface is represented by an intrinsic Lipschitz graph over $\mathbb M$. Let $\Sigma_{\mc U}$ be the family of Lipschitz surfaces in a bounded open set $\mc U\subset\G$.
Then the system $\mathbf E_{\mc U}$ is $p$-exceptional for $p\in (0,1)$.
\end{corollary}

In the next step we assume that the family of graphs is parametrised over a decomposition $\G=\mathbb M\cdot \mathbb H$ that belongs to the orbit of a subgroup $K\subset {\rm Iso}(\G)$ preserving the decomposition. Here we denote by ${\rm Iso}(\G)$ the group of grading preserving isometries of $\G$. 

\begin{theorem}\label{th:OrbitDecomposition}
Let $\G=\mathbb M\cdot \mathbb H$, and $K$ a subgroup of the group of isometries ${\rm Iso}(\G)$ preserving the decomposition. 
Let $\Sigma$ be a family of intrinsic $({\bf d_t},{\bf d_m})$-Lipschitz graphs over the orbit $K(\mathbb M)$ and $\Sigma_{\mc U}=\{S\cap\mc U:\ S\in \Sigma\}$. 
Then the system of measures $\mathbf E_{\mc U}$ is $p$-exceptional for $p\in (0,1)$.
\end{theorem}
\begin{proof}

By definition, for any $S\in \Sigma_{\mc U}$ there exists an intrinsic Lipschitz function $\hat f_S\colon\hat \Omega_S\to \hat{\mathbb H}$, where $\hat \Omega_S$ is an open set in the group $\hat{\M}$ such that $\M\cdot\mathbb H\in K(\hat{\M}\cdot \hat{\mathbb H})$. Then there is an isometric diffeomorphism $k\in K$ such that $\M\cdot\mathbb H=k(\hat{\M}\cdot \hat{\mathbb H})$. We write 
$$
f_S=k\circ \hat f_S\circ k^{-1}\colon \Omega_S\to \mathbb H,
$$
where $\Omega_S$ is an open subset in $\M$, such that $k(\hat \Omega_S)=\Omega_S$. Thus for any $S\in \Sigma_{\mc U}$ there exists an intrinsic Lipschitz function $f_S\colon\Omega_S\to \mathbb H$, such that 
\begin{itemize}
\item $\Omega_S$ is an open subset of $\mathbb M$;
\item $S=\graph{ f_S}$, i.e. $S=\Phi_S (\Omega_S)$,
where 
$$
\Phi_S(m)=k\big(\hat m\cdot \hat f_S(\hat m)\big)=k(\hat m)\cdot \Big(k\circ \hat f_S\circ k^{-1}\big(k(\hat m)\big)\Big), \quad m\in \Omega_S.
$$
\end{itemize}
We define $F\in L^p(\G,\mathbf g_{\G})$ as in~\eqref{May1 eq:3} and argue as in Theorem~\ref{th:1decomposition}. 
\end{proof}

\begin{corollary}
Let $\G=\mathbb M\cdot \mathbb H$ and $K\subset {\rm Iso}(\G)$. Let $\Sigma$ be a family of intrinsic $({\bf d_t},{\bf d_m})$-Lipschitz surfaces, such that locally each surface is represented by an intrinsic Lipschitz graph over an element of the orbit $K(\M)$ as in Theorem~\ref{th:OrbitDecomposition}.
Then $\mathbf E_{\mc U}$ is $p$-exceptional for $p\in (0,1)$.
\end{corollary}

Now we assume that the group $\G$ can be written as $\G=\mathbb M_{\alpha}\cdot \mathbb H_{\alpha}$ for finitely many $\alpha=1,2,\ldots, l$ and $\M_{\alpha}$ being $({\bf d_t},{\bf d_m})$-homogeneous non isomorphic subgroups for all $\alpha$. Under this assumption we state the following result.

\begin{theorem}\label{th:multipleDecomposition}
Let $\Sigma$ be a family of intrinsic $({\bf d_t},{\bf d_m})$-Lipschitz graphs where each graph is parametrised over one of the decompositions $\G=\mathbb M_{\alpha}\cdot \mathbb H_{\alpha}$, $\alpha=1,2,\ldots, l$. Let $\Sigma_{\mc U}=S\cap\,\mc U$, $S\in \Sigma$, and $\mathcal U\subset \G$ be an open set.
Then $M_p(\mathbf E_{\mc U})=0$ for $p\in (0,1)$.
\end{theorem}
\begin{proof}
We use the notation $\G_{\alpha}=\mathbb M_{\alpha}\cdot \mathbb H_{\alpha}$, $\alpha=1,2,\ldots, l$ and $\G^l=\G_1\times\ldots\times\G_l$. We define the selection map
$
\chi_{\alpha}(g)=m_{\alpha}h_{\alpha}$, $m_{\alpha}\in \mathbb M_{\alpha}$, $h_{\alpha}\in\mathbb H_{\alpha}$.
It can be considered as a composition of the map
\begin{equation*}
\begin{array}{cccccc}
P\colon&\G&\to &\G^l=&\G_1\times\ldots\times\G_l
\\
&g&\mapsto &&(m_{1}h_{1},\ldots, m_{l}h_{l});
\end{array}
\end{equation*}
followed by the projection on $\alpha$-slot. Then we define 
\begin{equation}\label{eq:F5}
F(g)=\sum_{\alpha=1}^l\phi_{r_{\alpha}}\circ T^{-1}_{\M_{\alpha}}\circ \Pi_{\M_{\alpha}}\circ\chi_{\alpha}(g).
\end{equation}
Then as in Theorem~\ref{th:1decomposition} we show that $F\in L^p(\G,\mathbf g_{\G})$. Moreover if 
$S\in \Sigma_{\mc U}(R,N)$ is the graph of a Lipschitz function $f_S\colon(\Omega_{\alpha})_S\to \mathbb H_{\alpha}$, for some $\alpha=1,\ldots,l$, where $(\Omega_{\alpha})_S$ is an open set in $\M_{\alpha}$, then $\Phi_S(m_{\alpha}):=m_{\alpha}\cdot f_S(m_{\alpha})$ the parametrization of $S$ associated with $f_S$. Then
\begin{eqnarray*}
N\int_{S\cap \mc U} F \, d\sigma_S
\geq  
\int_{\Phi_S^{-1}(\mc U)} F\circ \Phi_S (m_{\alpha})\, d\mathbf g_{\mathbb M_{\alpha}} 
\geq
\int_{B((\overline\zeta_{\alpha})_S, r_{\alpha})} \phi_{r_{\alpha}} (\overline\zeta_{\alpha})\, d\overline\zeta_{\alpha}=\infty,
\end{eqnarray*}
as in the proof of Theorem~\ref{th:1decomposition}. 
\end{proof}

\begin{corollary}
Let assume that $\G$ can be decomposed in a finitely many ways $\G=\mathbb M_{\alpha}\cdot \mathbb H_{\alpha}$, $\alpha=1,2,\ldots, l$, with non isomorphic subgroups $\mathbb M_{\alpha}$.

Let $\Sigma$ be a family of intrinsic $({\bf d_t},{\bf d_m})$-Lipschitz surfaces, such that locally each surface is represented by an intrinsic $({\bf d_t},{\bf d_m})$-Lipschitz graph over one of the groups $\mathbb M_{\alpha}$ as in Theorem~\ref{th:multipleDecomposition}.
Then $M_p(\mathbf E_{\mc U})=0$ for $p\in (0,1)$.
\end{corollary}

The last result in this section is a combination of Theorem~\ref{th:OrbitDecomposition} and Theorem~\ref{th:multipleDecomposition}. 
\begin{theorem}\label{th:multipleDecomposition-1}
Let $\Sigma$ be a family of intrinsic $({\bf d_t},{\bf d_m})$-Lipschitz surfaces where each graph is parametrised over one of the decompositions $\G=\mathbb M_{\alpha}\cdot \mathbb H_{\alpha}$, $\alpha=1,2,\ldots, l$. 
Moreover any term $\mathbb M_{\alpha}\cdot \mathbb H_{\alpha}$ is an element in the orbit of a subgroup $K_{\alpha}\subset {\rm Iso}(\G)$ preserving the decomposition $\mathbb M_{\alpha}\cdot \mathbb H_{\alpha}$.

Let $\Sigma_{\mc U}=S\cap\,\mc U$, $S\in \Sigma$, and $\mathcal U\subset \G$ be an open set.
Then $M_p(\mathbf E_{\mc U})=0$ for $p\in (0,1)$.
\end{theorem}

\begin{proof}
We define function $F\in L^{p}(\G,\mathbf g_{\G})$ as in~\eqref{eq:F5}.
Moreover for any point $q\in S\cap \mathcal U$ there is a neighbourhood $V$, such that for $S\cap \mathcal U\cap V$ there exists an intrinsic Lipschitz function $\hat f_S\colon(\hat{\Omega}_S)_{\alpha}\to \hat{\mathbb H}_{\alpha}$, where $(\hat \Omega_S)_{\alpha}\subset \hat\M_{\alpha}$ such that $\M_{\alpha}\cdot\mathbb H_{\alpha}\in K_{\alpha}(\hat\M_{\alpha}\cdot \hat{\mathbb H}_{\alpha})$. By choosing an isometry $k_{\alpha}\in K_{\alpha}$ we
find an intrinsic Lipschitz function $f_S\colon(\Omega_S)_{\alpha}\to \mathbb H_{\alpha}$, such that 
\begin{itemize}
\item $(\Omega_S)_{\alpha}$ is an open subset of $\mathbb M_{\alpha}$, $k_{\alpha}\big((\hat{\Omega}_S)_{\alpha}\big)=(\Omega_S)_{\alpha}$;
\item $S\cap \mathcal U\cap V=\graph{f_S}$, i.e. $S=\Phi_S \big((\Omega_S)_{\alpha}\big)$,
where 
$$
\Phi_S(m)=k_{\alpha}(\hat m)\cdot \Big(k_{\alpha}\circ \hat f_S\circ k^{-1}_{\alpha}\big(k_{\alpha}(\hat m)\big)\Big), \quad m\in (\Omega_S)_{\alpha},\quad \hat m\in (\hat \Omega_S)_{\alpha}.
$$
\end{itemize}
Then we proceed in the same way as in Theorem~\ref{th:1decomposition} and show that 
$$
N\int_{S\cap \mc U\cap V} F \, d\sigma_S
\ge  \int_{B(\zeta_S, R)} \phi_R (\zeta)\, d\zeta=\infty.
$$
\end{proof}

\subsubsection{Exceptional families for $p\geq 1$}\label{subsec:p1}


B. Fuglede proved that the system of $k$-dimensional Lipschitz surfaces in $\mathbb R^n$ which pass through a given point is $p$-exceptional if and only if $kp\leq n$~\cite{MR97720}. Here we show the sufficient part of the analogous statement for Carnot groups.

\begin{theorem}\label{prop:4}
Let $\Sigma\subset \Sigma^{({\bf d_t},{\bf d_m})}$ be a collection of intrinsic $({\bf d_t},{\bf d_m})$-Lipschitz graphs.
Suppose that all the graphs $S\in \Sigma$ contain a common point $g_0\in\mathbb G$. Then for ${\bf d_m} p\leq Q$ we have $M_p(\Sigma)=0$.
\end{theorem}

\begin{proof} Let $\Sigma\subset\Sigma^{({\bf d_t},{\bf d_m})}$ be a collection of intrinsic $({\bf d_t},{\bf d_m})$-Lipschitz graphs containing a common point $g_0\in\G$. We can assume that $g_0=e\in\mathbb G$ by the translation invariance of measures and the fact that the translation of an intrinsic Lipschitz graph is still an intrinsic Lipschitz graph, see~\cite[Theorem 3.2]{MR3511465}. 
We need to find a non-negative measurable function $F\colon\mathbb G\to\mathbb R$ such that $\int_{\mathbb G}F^p\,d\mathbf g_{\G}<\infty$, but $\int_SF\,d\sigma_S=\infty$ for any $S\in \Sigma$. 

Let $\|\cdot\|$ be a homogeneous norm on $\G$, for instance one of the types $(D_2)-(D_4)$ and let $d_\rho$ be a metric associated with the norm $\|\cdot\|$. We set
\begin{equation}\label{eq:case1}
F(g)=
\begin{cases}
\|g\|^{-{\bf d_m}},\quad&\text{if}\quad \|g\|<1,
\\
0,\quad&\text{if}\quad \|g\|\geq 1,
\end{cases}
\qquad {\bf d_m} p< Q.
\end{equation}
Then $F\in L_p(\mathbb G,\mathbf g_{\G})$ since 
$$
\int_{B_{d_\rho}(e,1)}|F|^p\,d\mathcal L^N=\omega\int_{0}^1r^{-p{\bf d_m}+Q-1}dr<\infty
$$ 
for $-p{\bf d_m}+Q>0$, 
where $\omega$ is a suitable constant depending only on $\|\cdot\|$, see, e.g.~\cite[Proposition 1.15]{MR657581}.

Consider intersections $S\cap B_{d_\rho}(e,\frac{1}{2^j})$, $j\in\mathbb N$. We divide the ball $B_{d_\rho}(e,1)$ into the spherical rings $R_j=B_{d_\rho}(e,\frac{1}{2^j})\setminus \overline B_{d_\rho}(e,\frac{1}{2^{j+1}})$. In each ring $R_j$ we choose a point $p_j\in S\cap\partial B_{d_\rho}(e,\frac{3}{2^{j+2}})$, Then $B_{d_\rho}(p_j,\frac{1}{2^{j+3}})\subset R_j$. We observe that $2^{(j+1){\bf d_m}}\geq\|g\|^{-{\bf d_m}}>2^{j{\bf d_m}}$ for $g\in R_j$. Then
\begin{eqnarray*}
\int_SF\,d\sigma_S
&\geq &
\int_{B_{d_\rho}(e,1)\cap S}\|g\|^{-{\bf d_m}}\,d\sigma_S
=
\sum_j\int_{R_j\cap S}\|g\|^{-{\bf d_m}}\,d\sigma_S
\\
&>&
\sum_j 2^{j{\bf d_m}}\mc S^{{\bf d_m}}_{d_\rho}(R_j\cap S)
>
\sum_j 2^{j{\bf d_m}}\mc S^{{\bf d_m}}_{d_\rho}(B_{d_\rho}(p_j,2^{-j-3})\cap S)
\\
&\geq&
c_12^{-3{\bf d_m}}\sum_j 2^{j{\bf d_m}}2^{-j{\bf d_m}}=\infty.
\end{eqnarray*}

If ${\bf d_m} p=Q$, then we need to change the function $F$ to
\begin{equation}\label{eq:case2}
F(g)=
\begin{cases}
\|g\|^{-{\bf d_m}}\big(\ln\frac{2}{\|g\|}\big)^{-\alpha},\quad&\text{if}\quad \|g\|<1,
\\
0,\quad&\text{if}\quad \|g\|\geq 1.
\end{cases}
\end{equation}
If we choose $\alpha\in[\frac{{\bf d_m}}{Q},1]$, then $F\in L^p(\mathbb G,\mathbf g_{\G})$ and $\int_SF(g)\,d\sigma_S=\infty$. Indeed, 
$$
\int_{B_{d_\rho}(0,1)}|F|^p\,d\mathbf g_{\G}=\omega\int_{0}^1r^{-p{\bf d_m}+Q-1}\big(\ln\frac{2}{r}\big)^{-\alpha p}dr=\omega\int_{\ln 2}^{\infty}t^{-\alpha p}dt<\infty
$$
if $\alpha p>1$ or equivalently $\alpha>\frac{{\bf d_m}}{Q}$. From the other side 
\begin{eqnarray*}
\int_SF\,d\sigma_S
&\geq &
\int_{R_j\cap S}\|g\|^{-{\bf d_m}}\Big(\ln\frac{2}{\|g\|}\Big)^{-\alpha}\,d\sigma_S
\\
&>&
\sum_j 2^{j{\bf d_m}}(\ln 2^{j+2})^{-\alpha}\mc S^{{\bf d_m}}_{d_\rho}(B_{d_\rho}(p_j,2^{-j-3})\cap S)
\\
&\geq&
c_12^{-3{\bf d_m}}(\ln 2)^{-\alpha}\sum_j (j+2)^{-\alpha}=\infty
\end{eqnarray*}
for $\alpha<1$.
\end{proof}

\begin{corollary}
Let $\Sigma\subset \Sigma^{({\bf d_t},{\bf d_m})}$ be a collection of intrinsic $({\bf d_t},{\bf d_m})$-Lipschitz surfaces.
Suppose that all the surfaces $\widehat S\in \Sigma$ contain a common point $g_0\in\mathbb G$. Then for ${\bf d_m} p\leq Q$ we have $M_p(\Sigma)=0$.
\end{corollary}

\begin{proof} We assume that $g_0=e$ and, as in Theorem~\ref{prop:4}, consider the  function $F\in L^p(\G,\mathbf g_{\G})$ for ${\bf d_m} p\leq Q$. To show that $\int_{\widehat S}F\,d\sigma_{\widehat S}=\infty$, we fix a surface $\widehat S\in\Sigma$ and a neighbourhood $U$ of $e$, such that $S=\widehat S\cap U$ is an intrinsic Lipschitz graph. 
Then we argue as in Theorem~\ref{prop:4}.
\end{proof}

In the following sections we will study $p$-exeptional sets of intrinsic Lipschitz surfaces passing through a common point when ${\bf d_m}p>Q$.  We will make special constructions of intrinsic Lipschitz surfaces that will reveal the situation of being  $p$-exceptional. The first step to the construction is the study of an analog of a Grassmann manifold on special type of Carnot groups.


\section{Orthogonal Grassmannians}\label{secOrthGrassmannians}


In this section we construct orthogonal Grassmannians of Lie subalgebras on specially chosen H-type Lie algebras. We start from the overview of Grassmannians of $k$-plains of $n$-dimensional Euclidean space, reminding that they are orbits under the action (modulo the isotropy subgroup) of the isometry group $O(n)$ of the Euclidean space, see Section~\ref{sec:EuclGr}. In order to make a proper construction of the Grassmannian of subalgebras we first remind in Section~\ref{sec:IsomGroup} the structure of the isometry group $\Iso$ of H-type algebras. Then we make construction of orthogonal Grassmannians of subalgebras and study their properties in Section~\ref{sec:HGr}


\subsection{Overview over the Grassmannians in $\mathbb K^n$}\label{sec:EuclGr}



\subsubsection{Definition of the groups $\Orth(n)$, $\U(n)$, and $\Sp(n)$}


The orthogonal group $\Orth(n)$ in $\mathbb R^n$, endowed with the standard Euclidean inner product $\langle .\,,.\rangle_{\mathbb R}$, is
$$
\Orth(n)=\{A\colon \mathbb R^n\to \mathbb R^n:\ \langle Av,Av\rangle_{\mathbb R}=\langle v,v\rangle_{\mathbb R}\}.
$$
The unitary group $\U(n)$ acting in $\mathbb C^n$, endowed with the standard Hermitian inner product $\langle .\,,.\rangle_{\mathbb C}$, is defined analogously by
$$
\U(n)=\{A\colon \mathbb C^n\to \mathbb C^n:\ \langle Av,Av\rangle_{\mathbb C}=\langle v,v\rangle_{\mathbb C}\}.
$$
Finally, the quaternion unitary group (or compact symplectic group) $\Sp(n)$ acting in right quaternion space $\mathbb Q^n$, is defined by
$$
\Sp(n)=\{A\colon \mathbb Q^n\to \mathbb Q^n:\ \langle Av,Av\rangle_{\mathbb Q}=\langle v,v\rangle_{\mathbb Q}\}.
$$
where $\langle .\,,.\rangle_{\mathbb Q}$ is a quaternion Hermitian product in $\mathbb Q^n$, see, for instance~\cite{rossmann2006lie}.

Let us write $\mathbb K$ for the division algebras $\mathbb R$, $\mathbb C$, or $\mathbb Q$ and 
$\U(n,\mathbb K)$ for the groups $\Orth(n)$, $\U(n)$, and $\Sp(n)$, respectively. We let $k$ be an integer satisfying $0<k\leq n$. The Grassmann manifold or Grassmannian $\Gr_k(\mathbb K^n)$ is defined as the set of $k$-dimensional vector subspaces in $\mathbb K^n$:
$$
\Gr_k(\mathbb K^n)=\{V\ \text{is a}\  k-\text{dimensional vector subspace of} \ \ \mathbb K^n\}.
$$
Note that the vector space $\mathbb Q^n$ is defined as the {\it right vector space} with respect to the right multiplication by scalars from $\mathbb Q$. The same agreement is done for the subspaces $V\subset \mathbb Q^n$.

The group $\U(n,\mathbb K)$ acts transitively on the set $\Gr_k(\mathbb K^n)$ via
$$
A.V=\{Av\in\mathbb K^n:\ v\in V,\ A\in \U(n,\mathbb K)\}.
$$
Fix a plain $\hat V\in \Gr_k(\mathbb K^n)$. Let $K_{\hat V}=\{A\in \U(n,\mathbb K):\ A.\hat V=\hat V\}$ be the isotropy group of $\hat V$. It follows that the Grassmannian $\Gr_k(\mathbb K^n)$ admits the structure of a compact manifold~\cite{Warner94} given by
$$
\Gr_k(\mathbb K^n)=\U(n,\mathbb K)/K_{\hat V},
$$
where $K_{\hat V}$ is isomorphic to $\U(k,\mathbb K)\times\U(n-k,\mathbb K)$. Note that there is a diffeomorphism
$
\Gr_{n-k}(\mathbb K^n)\cong \Gr_k(\mathbb K^n)
$
mapping any $V\in \Gr_{n-k}(\mathbb K^n)$ to its orthogonal complements $V^{\bot}\in \Gr_{k}(\mathbb K^n)$.


\subsubsection{Measure on the Grassmannians}


Let us remind the definition of a measure on the Grassmannian $\Gr_{k}(\mathbb K^n)$. The continuous map
$$
\psi\colon \U(n,\mathbb K)\to \Gr_{k}(\mathbb K^n),
$$
which is the composition of the projection map to the quotient and the diffeomorpism giving the manifold structure to $\Gr_{k}(\mathbb K^n)$ is used to push-forward a measure from $\U(n,\mathbb K)$ to $\Gr_{k}(\mathbb K^n)$. We take for granted that the group $\U(n,\mathbb K)$ carries bi-invariant normalised measure $\lambda$:
$$
\lambda (AU)=\lambda (UA)=\lambda (U),
\qquad\lambda(\U(n,\mathbb K))=1,
$$
for any Borel set $U\subset \U(n,\mathbb K)$ and any isometry $A\in \U(n,\mathbb K)$. The measure $\mu$ on the Borel sets $\Omega\subset\Gr_{k}(\mathbb K^n)$ is defined as a pushforward of $\lambda$:
$$
\mu(\Omega)=(\psi_{\sharp}\lambda)(\Omega)=\lambda(\psi^{-1}(\Omega))=\lambda\{A\in \U(n,\mathbb K):\ V=A.\hat V\in \Omega\}. 
$$
It can be verified that  $\mu$ is normalised and it is $\U(n,\mathbb K)$-invariant.
The converse is also true: a normalised $\U(n,\mathbb K)$-invariant measure on the homogeneous space $\Gr_{k}(\mathbb K^n)$ is a push forward of the normalised Haar measure from $\U(n,\mathbb K)$, see for instance~\cite{MR1333890}.

Note that an $(n-1)$-dimensional sphere $S(0,R)$ in $\mathbb R^n$ can be considered as a particular case of the Grassmanniann $\Gr_{n-1}(\mathbb R^n)$. If we denote by $K_{x}(\mathbb R)$ an isotropy group of a point $x\in S(0,R)$ under the action of $\Orth(n)$, then the following manifolds are diffeomorphic 
$$
S(0,R)\sim \Gr_{n-1}(\mathbb R^n)\sim \Orth(n)/K_{x}(\mathbb R).
$$ 
The push-forward of the normalised measure $\lambda$ on $\Orth(n)$ to $S(0,R)$ coincides with the normalised surface measure on $S(0,R)$, see~\cite[Theorem 3.7]{MR1333890}. 
%


\subsection{Isometry groups of special $H$-type Lie algebras}\label{sec:IsomGroup}


Before we make the construction of orthogonal Grassmannians on some special $H$-type Lie algebras, we describe the group of isometries of these Lie algebras.

Recall the definition of an $H$-type Lie algebra $\h=(\h_1\otimes\h_2,[.\,,.],\langle.\,,.\rangle_{\mathbb R})$ from Section~\ref{sec:Htypegroups}.  
The group of isometries $\Iso(\h)$ of $H$-type Lie algebras were studied in~\cite{Riehm82,Riehm84}. It was shown that 
$$
\Iso(\h)=\{(A,C)\in \Orth(\h_1)\times \Orth(\h_2):\ A^tJ_zA=J_{C^t(z)}\quad\text{for any}\quad z\in\h_2\}.
$$
The group $\Iso(\h)$ is isogenous to the product of the Pin group $\Pin(\h_2,\langle.\,,.\rangle_{\R})$ of the Clifford algebra $\Cl(\h_2,\langle.\,,.\rangle_{\R})$ and a classical group $\mathbb A$. The latter means that there is  
a surjective morphism 
\begin{equation}\label{eq:form-of-isometry}
\begin{array}{ccc}
\phi\colon \Pin(\h_2,\langle.\,,.\rangle_{\R})\times \mathbb A&\to& \Iso(\h)\subset\Orth(\h_1)\times \Orth(\h_2)
\\
\theta=(\alpha,A)&\mapsto &\phi(\theta)=\big(J_{\alpha}\circ A,\kappa(\alpha)\big).
\end{array}
\end{equation}
with a finite kernel of order 2 or 4. Here 
\begin{equation}\label{eq:rho}
\kappa\colon \Pin(\h_2,\langle.\,,.\rangle_{\R})\to \Orth(\h_2)
\end{equation}
is the double cover of the orthogonal group defined by 
$$
\kappa(\alpha)v=\alpha v\alpha^{-1},\quad v\in \h_2,\ \ \alpha\in \Pin(\h_2,\langle.\,,.\rangle_{\R}).
$$
We will not give the full description of the group of isometries, but rather concentrate on the cases when the restriction $\Iso(\h)\vert_{\h_1}$ on $\h_1$ acts transitively on the spheres $S(0,r)\in \h_1$ and the vector space $\h_1$, considered as a Clifford module, is not irreducible. Only in these cases the construction of the Grassmannian on the Lie algebra $\h$ is not trivial. Thus we will consider the following $H$-type Lie algebras:
\begin{itemize}
\item[RH:]{The Heisenberg algebra $\h_{\mathbb R}^{n}=(\mathbb R^{2n+1},[.\,,.],\langle.\,,.\rangle_{\mathbb R})$, $n>1$ with $\h_{\mathbb R}=\h_1\oplus\h_2\cong\R^{2n}\otimes\R$;}
\item[CH:]{The complex Heisenberg algebra $\h_{\mathbb C}^{n}=(\mathbb R^{4n+2},[.\,,.],\langle.\,,.\rangle_{\mathbb R})$, $n>1$ with $\h_{\mathbb C}=\h_1\oplus\h_2\cong\R^{4n}\otimes\R^2$;}
\item[QH:]{The quaternion Heisenberg algebra $\h_{\mathbb Q}^{n}=(\mathbb R^{4n+3},[.\,,.],\langle.\,,.\rangle_{\mathbb R})$, $n>1$ with $\h_{\mathbb Q}=\h_1\oplus\h_2\cong\R^{4n}\otimes\R^3$.}
\end{itemize}
The commutators in all the cases are defined by~\eqref{eq:commutator} and the scalar products are the standard Euclidean products making the direct sums orthogonal. The names real, complex, and quaternion are attached to the names of the algebras by the following reasons. For the Lie algebra $\h_{\mathbb R}^{n}$, the group of real symplectic transformations $\Sp(n,\mathbb R)$ is a subgroup of automorphisms, leaving the center $\h_2$ of the Lie algebra $\h_{\mathbb R}^{n}$ invariant. For the Lie algebra $\h_{\mathbb C}^{n}$, the group of complex symplectic transformations $\Sp(2n,\mathbb C)$ is a subgroup of the Lie algebra automorphisms, leaving the center invariant and also $\h_{\mathbb C}^{n}$ is isomorphic to the complexification of $\h_{\mathbb R}^{n}$. For $\h_{\mathbb Q}^{n}$ the group of quaternion unitary transformations $\Sp(n)$ is a subgroup of the Lie algebra automorphisms, leaving the center invariant. The center $\h_2$ of $\h_{\mathbb Q}^{n}$ is isomorphic to the space of pure imaginary quaternions. 

Now we describe the isometry groups in each case.


\subsubsection{The isometry group $\Iso(\h_{\mathbb R})$}\label{IsohR} 


Recall the definition of the Heisenberg algebra in Section~\ref{sec:Heis-R}. We write $\h_{\mathbb R}=\h_1\oplus\h_2$. Let $\epsilon\in\h_2\cong\mathbb R$ be a vector satisfying $\langle \epsilon,\epsilon\rangle_{\mathbb R}=1$. 
The vector space $\h_1\cong\R^{2n}$ carries natural almost complex structure given by the action of $J_\epsilon$, $\epsilon\in \h_2$. Therefore, the space $\h_1\cong\mathbb R^{2n}$ can be considered as a complex vector space $\h_1\cong\mathbb C^{n}$, where the multiplication by a complex number $\alpha\in\mathbb C$ is defined by 
$
\alpha v=(a+ib)v=av+bJ_{\epsilon} v$, $v\in\h_1$.
The maximal compact subgroup $\mathbb A=\U(n)\subset \Sp(2n,\mathbb R)$ is an isometry on $\h_1\cong\R^{2n}$. 
The corresponding Hermitian form $\langle .\vert.\rangle_{\mathbb C}$ on $\h_1\cong\mathbb C^n$ is defined by making use of the real scalar product $\langle .\,,.\rangle_{\mathbb R}$ by the following
\begin{equation}\label{eq:complexHermitian}
\langle u\vert v\rangle_{\mathbb C}=\langle u,v\rangle_{\mathbb R}-i\langle J_{\epsilon}u,v\rangle_{\mathbb R}.
\end{equation}
The imaginary part of the Hermitian product~\eqref{eq:complexHermitian} defines the symplectic form $\omega\colon\h_1\times\h_1\to\mathbb R$, that is compatible with the real inner product on $\h_1$. Namely
\begin{equation}\label{eq:symplectic-form}
\omega(u,v)=\langle J_{\epsilon}u,v\rangle_{\mathbb R}=
\langle [u,v],\epsilon\rangle_{\mathbb R}.
\end{equation}
which implies
$
\omega(J_{\epsilon}u,J_{\epsilon}v)=\omega(u,v), \quad
\omega(u,J_{\epsilon}v)=\langle u,v\rangle_{\mathbb R}.
$

We use the basis described in Section~\ref{sec:Heis-R} and give the coordinates to $\h^n_{\R}$. Consider the product of two spheres on $\h^n_{\R}\cong \mathbb R^{2n}\times\mathbb R$ centered at the origin
\begin{eqnarray}\label{eq:polysphereHeis}
\mathcal S_{\mathbb R}(0,r_1,r_2)&=&S^{2n-1}(0,r_1)\times S^0(0,r_2)\nonumber
\\
&=&
\{v\in\h_1:\ \langle v,v\rangle_{\mathbb R}=r_1^2\}\times\{z\in\h_2, \langle z,z\rangle_{\mathbb R}=r_2^2\}
\end{eqnarray}
\begin{lemma}\label{lem:transit-hr}
The group $\Iso(\h_{\mathbb R}^n)$ acts transitively on $\mathcal S_{\R}(0,r_1,r_2)$. 
\end{lemma}
\begin{proof}
The map $\kappa(\epsilon)\in \Orth(\h_2)$, defined in~\eqref{eq:rho} acts either as a reflection with respect to the origin of $\h_2$ or as the identity, since $\h_2\cong\mathbb R$. 
%
Let $(x_1,z), (x_2,-z)\in\mathcal S_{\R}(0,r_1,r_2)$. Then we find a map $(J_{\epsilon}\circ A,\kappa(\epsilon))\in\Iso(\h^n_{\mathbb R})$ that maps the point $(x_1,z)$ to $(x_3,-z)=(J_{\epsilon}\circ A x_1,\kappa(\epsilon)z)$, $A\in\U(n)$. Since the subgroup $(\U(n)\times\Id_{\h_2})\subset \Iso(\h^n_{\mathbb R})$ acts transitively on $S^{2n-1}(0,r_1)$ we can find the transformation mapping $(x_3,-z)$ to $(x_2,-z)$. 
\end{proof}

At the end we recall that the basis of left invariant vector fields on the Heisenberg group $H_{\R}^n$ are
\begin{equation}\label{eq:libasisR}
X_{1k}=\frac{\partial}{\partial x_{1k}}-\frac{x_{2k}}{2}\frac{\partial}{\partial \epsilon_1},
\quad
X_{2k}=\frac{\partial}{\partial x_{2k}}+\frac{x_{1k}}{2}\frac{\partial}{\partial \epsilon_2},
\end{equation}
for $k=1,\ldots,n$.

\subsubsection{Isometry groups of H-type Lie algebra $\mathfrak{h}^{n}_{\CC}$}\label{sec:ComplexifiedHeis} 


Let $\h_1\cong\mathbb R^{4n}\subset \mathfrak{h}^{n}_{\mathbb C}$ be the horizontal subspace and $\h_2\cong\mathbb R^2\subset \mathfrak{h}^{n}_{\mathbb C}$ be the vertical subspace. Then $\Cl(\h_2,\langle.\,,.\rangle_{\mathbb R})\cong\Cl(\mathbb R^2)$ contains two elements $\epsilon_1,\epsilon_2$ such that 
$$
\epsilon_1^2=\epsilon_2^2=-1,\quad \langle \epsilon_1,\epsilon_1\rangle_{\mathbb R}=\langle\epsilon_2,\epsilon_2\rangle_{\mathbb R}=1,\quad \langle\epsilon_1,\epsilon_2\rangle_{\mathbb R}=0,
$$ 
and moreover $\epsilon_3=\epsilon_1\epsilon_2$ satisfies
$
\epsilon_3^2=-1$. 

We first consider the case $n=1$, that is $\h_1\cong\mathbb R^{4}\subset \mathfrak{h}^{1}_{\mathbb C}$.
The maps $J_{\epsilon_i}\colon \h_1\to \h_1$, $i=1,2$, are orthogonal transformations. A convenient orthonormal basis for $\h_1$ can be constructed by the following. We choose a vector $v\in\h_1$, such that $\langle v\,,v\rangle_{\R}=1$ and define the orthonormal vectors
\begin{equation}\label{eq:Basis-1}
X_1=v,\quad X_2=J_{\epsilon_1} v,\quad X_3=J_{\epsilon_2} v,\quad X_4=J_{\epsilon_2}J_{\epsilon_1} v.
\end{equation}
The commutation relations according to~\eqref{eq:commutator} are
\begin{equation}\label{eq:ComhC}
[X_1,X_2]=-[X_3,X_4]=\epsilon_1,\quad
[X_1,X_3]=[X_2,X_4]=\epsilon_2.
\end{equation}
We show now that $\h^1_{\CC}$ is the complexified Lie algebra of $\h^1_{\R}$.
Let $X_1,X_2=J_{\epsilon}X_1$ be an orthonormal basis of $\h_1\subset \h_{\R}^1$ and $\h_2=\spn_{\R}\{\epsilon\}$ is the center of the Lie algebra $\h^1_{\R}$. 
Then the complexification $\h_1^{\CC}=\mathbb C\otimes\h_1$ of the vector space $\h_1\subset\h^1_{\R}$ can be described as any of the following direct sums
\begin{eqnarray}\label{eq:Complexif}
\h_1^{\CC}&=&\spn_{\R}\{Z=X_1+iX_2\}\oplus\spn_{\R}\{\bar Z=X_1-iX_2\}\nonumber
\\
&=&
\spn_{\R}\{X_1,X_2\}\oplus\spn_{\R}\{iX_1,iX_2\}
\\
&=&
\spn_{\CC}\{X_1,X_2\}.\nonumber
\end{eqnarray}
The Lie bracket on $\h_1^{\CC}$ is a complex linear Lie bracket on $\h_1$: \begin{equation}\label{eq:commutator-complex}
[X_1,X_2]=-[iX_1,iX_2]=\epsilon,\quad [X_1,iX_2]=[iX_1,X_2]=i\epsilon.
\end{equation}
Thus the center $\h_2^{\CC}$ of the complexified Heisenberg algebra is given by
\begin{eqnarray}\label{eq:Complexif-2}
\h_2^{\CC}=\spn_{\CC}\{\epsilon\}=\spn_{\R}\{\epsilon, i\epsilon\}.
\end{eqnarray}
Recall that the real Heisenberg algebra has the complex structure $J_{\epsilon}\colon\h_1\to\h_1$ defined by $J_{\epsilon}X_1=X_2$, $J_{\epsilon}X_2=-X_1$.
We extend $J_{\epsilon}$ to $\h_1^{\CC}$ by linearity, meaning that $J_{\epsilon}(iu)=iJ_{\epsilon}(u)$ for $u\in\spn_{\R}\{X_1,X_2\}$. We define another complex structure 
$$
J_{i\epsilon}\colon\h_1^{\CC}\to\h_1^{\CC}:\quad
J_{i\epsilon}X_1=iX_2,\quad J_{i\epsilon}X_2=iX_1.
$$
It is easy to check that $J_{\epsilon}J_{i\epsilon}=-J_{i\epsilon}J_{\epsilon}$.
Note that if we denote 
\begin{eqnarray}\label{eq:renotation}
&\epsilon_1=\epsilon,\quad \epsilon_2=i\epsilon,
\\
&X_1=X_1,\quad X_2=J_{\epsilon_1}X_1,\quad
X_3=iX_2=J_{\epsilon_2}X_1,\quad
X_4=iX_1=J_{\epsilon_2}J_{\epsilon_1}X_1,\nonumber
\end{eqnarray}
then we recover basis~\eqref{eq:Basis-1} and commutation relations~\eqref{eq:ComhC} of $\h_{\CC}^1$.

We show now that the horizontal space $\h_1\subset\h_{\CC}^1$ of the complexified Heisenberg algebra is a complex symplectic space.  
The real symplectic form $\omega$ defined in~\eqref{eq:symplectic-form} can be extended to the complex symplectic form $\omega^{\CC}\colon \h_1\times{\h_1}\to\CC$, $\h_1\subset\h_{\CC}^1$, by
\begin{equation}\label{eq:complexSymplectic}
\omega^{\CC}(u+iv,x+iy)=\omega(u,x)+\omega(v,y)+i\big(\omega(v,x)+\omega(y,u)\big).
\end{equation}
Here we consider $\h_1\subset\h_{\CC}^1$ as a complex space~\eqref{eq:Complexif}. Then $\omega^{\CC}$ satisfies
$$
\omega^{\CC}(z_1,z_2)=-\overline{\omega^{\CC}(z_2,z_1)},
$$
$$
\omega^{\CC}(\alpha z_1,z_2)=\alpha\omega^{\CC}(z_1,z_2),\quad
\omega^{\CC}(z_1,\alpha z_2)=\bar\alpha\omega^{\CC}(z_1,z_2).
$$
The symplectic basis  $X_1,X_2=J_{\epsilon_1}X_1$ over $\R$ for the real symplectic vector space $\h_1\subset\h_{\R}^1$ is the symplectic basis over $\CC$ for the complex symplectic vector space $\h_1\subset\h^1_{\CC}$. 

The multidimensional algebra $\mathfrak{h}^{n}_{\mathbb C}$ as a vector space is the Cartesian product
\begin{equation}\label{eq:direct-product}
\mathfrak{h}^{n}_{\mathbb C}=\big(\h_1^{\CC}\big)_1\times\ldots\times\big(\h_1^{\CC}\big)_n\times\h_2^{\CC},
\end{equation} 
where $\h_1^{\CC}$ is defined in~\eqref{eq:Complexif} and $\h_2^{\CC}$ is defined in~\eqref{eq:Complexif-2}.
The commutation relations are given by~\eqref{eq:commutator-complex}, if the vectors belong to the same slott $\big(\h_1^{\CC}\big)_k$ in the Cartesian  product~\eqref{eq:direct-product} and zero otherwise. The real scalar product $\langle.\,,.\rangle_{\mathbb R}$ is extended to the Cartesian product~\eqref{eq:direct-product} by making the different slots orthogonal. 
The multidimensional algebra $\mathfrak{h}^{n}_{\mathbb C}$ has real topological dimension $4n+2$ and the Hausdorff dimension $4n+4$. The corresponding left invariant vector fields on the complexified Heisenberg group $H_{\CC}^n$ are 
\begin{equation}\label{eq:libasisC}
\begin{array}{ccccccc}
X_{1k}&=&\frac{\partial}{\partial x_{1k}}-\frac{x_{2k}}{2}\frac{\partial}{\partial \epsilon_1}
-\frac{x_{3k}}{2}\frac{\partial}{\partial \epsilon_2},
\\
X_{2k}&=&\frac{\partial}{\partial x_{2k}}+\frac{x_{1k}}{2}\frac{\partial}{\partial \epsilon_1}
-\frac{x_{4k}}{2}\frac{\partial}{\partial \epsilon_2},
\\
X_{3k}&=&\frac{\partial}{\partial x_{3k}}+\frac{x_{4k}}{2}\frac{\partial}{\partial \epsilon_1}
+\frac{x_{1k}}{2}\frac{\partial}{\partial \epsilon_2},
\\
X_{4k}&=&\frac{\partial}{\partial x_{4k}}-\frac{x_{3k}}{2}\frac{\partial}{\partial \epsilon_1}
+\frac{x_{2k}}{2}\frac{\partial}{\partial \epsilon_2},
\end{array}
\qquad k=1,\ldots, n.
\end{equation}

The subgroup $\mathbb A\cong\Sp(n)$ of the isometry group $\Iso(\h^n_{\mathbb C})$ preserves the quaternion Hermitian product $\langle.\vert.\rangle_{\mathbb Q}$ on $\h_1\subset\h^n_{\CC}$, defined by
\begin{equation}\label{eq:quaternionHermitian}
\langle z\vert w\rangle_{\mathbb Q}=\langle z, w\rangle_{\mathbb R}
-i\langle J_{\epsilon_1}z, w\rangle_{\mathbb R}
-j\langle J_{\epsilon_2}z, w\rangle_{\mathbb R}
-k\langle J_{\epsilon_2}J_{\epsilon_1}z, w\rangle_{\mathbb R}.
\end{equation}
The group $\mathbb A\cong\Sp(n)$ acts transitively on the unit sphere $S^{4n-1}\subset\h_1\cong \mathbb C^{2n}$, where $\h_1\subset \mathfrak{h}^{n}_{\mathbb C}$. 


\subsubsection{Isometry groups of H-type Lie algebra $\mathfrak{h}^{n}_{\mathbb Q}$} 


Since $\mathbb R^3\cong \h_2\subset \mathfrak{h}^{1}_{\mathbb Q}$, there is three linearly independent length one elements $\epsilon_i\in \Cl(\mathbb R^3,\langle.\,,.\rangle_{\mathbb R})$, $i=1,2,3$, satisfying the quaternion relations
\begin{equation}\label{eq:quaternion-relation}
\epsilon_1^2=\epsilon_3^2=\epsilon_3^2=\epsilon_1\epsilon_2\epsilon_3=-1.
\end{equation}
We introduce a quaternion structure on $\mathbb R^4\cong \h_1\subset \mathfrak{h}^{1}_{\mathbb Q}$ by defining the multiplication by a quaternion number $q=a+ib+jc+kd\in\mathbb Q$ as follows
$$
q v=(a+ib+jc+kd)v=av+bJ_{\epsilon_1} v+cJ_{\epsilon_2} v+dJ_{\epsilon_3} v,\quad v\in\h_1.
$$
The quaternion Hermitian product is 
\begin{equation}\label{eq:Hermitian-product1}
\langle u\vert v\rangle_{\mathbb Q}=\langle u,v\rangle_{\mathbb R}-i\langle J_{\epsilon_1}u,v\rangle_{\mathbb R}-j\langle J_{\epsilon_2}u,v\rangle_{\mathbb R}-k\langle J_{\epsilon_3}u,v\rangle_{\mathbb R}.
\end{equation}
To construct an orthonormal basis we choose $v\in \h_1$ with $\langle v,v\rangle_{\mathbb R}=1$
and set
\begin{equation}\label{eq:Basis-2}
X_1=v,\quad X_2=J_{\epsilon_1} v,\quad X_3=J_{\epsilon_2} v,\quad X_4=J_{\epsilon_3} v.
\end{equation}
The commutation relations are
\begin{eqnarray}\label{eq:commutator-quaternion}
&[X_1,X_2]=-[X_3,X_4]=\epsilon_1,\quad
[X_1,X_3]=[X_2,X_4]=\epsilon_2,
\\
&[X_1,X_4]=-[X_2,X_3]=\epsilon_3.\nonumber
\end{eqnarray}
The space $\h_1=\spn_{\mathbb R}\{X_1,X_2,X_3,X_4\}\subset \h^1_{\CC}$ is isomorphic and isometric to the 1-dimensional quaternion space endowed by the quaternion Hermitian product~\eqref{eq:Hermitian-product1}. 

We notice the relation between $\h^1_{\CC}$ and $\h^1_{\mathbb Q}$. Due to~\eqref{eq:quaternion-relation}, the action of $J_{\epsilon_3}$ can be expressed as $J_{\epsilon_3}=J_{\epsilon_1}J_{\epsilon_2}$ and therefore the space $\h_1\subset\h^1_{\mathbb Q}$ can be considered as a complex symplectic space with respect to $\omega^{\CC}$ in~\eqref{eq:complexSymplectic}. In addition to that we add a real symplectic form 
$
\omega_3(u,v)=\langle J_{\epsilon_3}u,v\rangle_{\R}.
$

The multidimensional algebra $\mathfrak{h}^{n}_{\mathbb Q}$ as a vector space is the Cartesian product
\begin{equation}\label{eq:direct-product}
\mathfrak{h}^{n}_{\mathbb Q}=\big(\h_1\big)_1\times\ldots\times\big(\h_1\big)_n\times\h_2.
\end{equation} 
The commutation relations are given by~\eqref{eq:commutator-quaternion}, if the vectors belong to the same slott $\big(\h_1\big)_k$ in the Cartesian  product~\eqref{eq:direct-product} and zero otherwise.
The multidimensional algebra $\mathfrak{h}^{n}_{\mathbb Q}$ has real topological dimension $4n+3$ and the Hausdorff dimension $4n+6$. 
We note that the Clifford module $\big(\h_1\big)_1\times\ldots\times\big(\h_1\big)_n$ in this case is assumed to be isotypic, that corresponds the fact that the product $J_{\epsilon_1}J_{\epsilon_2}J_{\epsilon_3}$ acts as minus identity on every slot $\big(\h_1\big)_k$. The subgroup $\mathbb A\cong\Sp(n)$ of the isometry group $\Iso(\h_{\mathbb Q}^n)$ preserves the quaternion Hermitian product on $\h_1\subset\h^n_{\mathbb Q}$, defined in~\eqref{eq:Hermitian-product1}. 
The corresponding basis of left invariant vector fields on the group $H_{\mathbb Q}^n$ is 
\begin{equation}\label{eq:libasisQ}
\begin{array}{ccccccc}
X_{1k}&=&\frac{\partial}{\partial x_{1k}}-\frac{x_{2k}}{2}\frac{\partial}{\partial \epsilon_1}
-\frac{x_{3k}}{2}\frac{\partial}{\partial \epsilon_2}
-\frac{x_{4k}}{2}\frac{\partial}{\partial \epsilon_3}
\\
X_{2k}&=&\frac{\partial}{\partial x_{2k}}+\frac{x_{1k}}{2}\frac{\partial}{\partial \epsilon_1}
-\frac{x_{4k}}{2}\frac{\partial}{\partial \epsilon_2}
+\frac{x_{3k}}{2}\frac{\partial}{\partial \epsilon_3},
\\
X_{3k}&=&\frac{\partial}{\partial x_{3k}}+\frac{x_{4k}}{2}\frac{\partial}{\partial \epsilon_1}
+\frac{x_{1k}}{2}\frac{\partial}{\partial \epsilon_2}
-\frac{x_{2k}}{2}\frac{\partial}{\partial \epsilon_3},
\\
X_{4k}&=&\frac{\partial}{\partial x_{4k}}-\frac{x_{3k}}{2}\frac{\partial}{\partial \epsilon_1}
+\frac{x_{2k}}{2}\frac{\partial}{\partial \epsilon_2}
+\frac{x_{1k}}{2}\frac{\partial}{\partial \epsilon_3},
\end{array}
\qquad k=1,\ldots, n.
\end{equation}

We use the bases~\eqref{eq:Basis-1} and~\eqref{eq:Basis-2} to identify $\h^n_{\mathbb C}$ with $\mathbb R^{4n}\times\mathbb R^2$ and $\h^n_{\mathbb Q}$ with $\mathbb R^{4n}\times\mathbb R^3$. Define the product of spheres
\begin{eqnarray*}
\mathcal S_{\mathbb C}(0,r_1,r_2)&=&\{(x,z)\in\h^n_{\mathbb C}:\ x\in S^{4n-1}(0,r_1)\subset\h_1, z\in S^1(0,r_2)\subset \h_2\}
\\
&\cong& S^{4n-1}(0,r_1)\times S^1(0,r_2),\quad \mathcal S_{\mathbb C}(0,r_1,r_2)\subset \h^{n}_{\mathbb C}
\end{eqnarray*}
and 
\begin{eqnarray*}
\mathcal S_{\mathbb Q}(0,r_1,r_2)&=&\{(x,z)\in\h^n_{\mathbb Q}:\ x\in S^{4n-1}(0,r_1)\subset\h_1, z\in S^2(0,r_2)\subset \h_2\}
\\
&\cong &
S^{4n-1}(0,r_1)\times S^2(0,r_2),\quad \mathcal S_{\mathbb Q}(0,r_1,r_2)\subset \h^{n}_{\mathbb Q}.
\end{eqnarray*}

\begin{lemma}\label{lem:transit-hr-1}
The groups of isometries $\Iso(\h_{\mathbb C}^n)$ and $\Iso(\h_{\mathbb Q}^n)$ act transitively on the respective products of spheres $\mathcal S_{\mathbb C}(0,r_1,r_2)$ and $\mathcal S_{\mathbb Q}(0,r_1,r_2)$. 
\end{lemma}

\subsection{Grassmannians on special $H$-type Lie algebras}\label{sec:HGr}


We will construct orthogonal Grassmannians on the H-type Lie algebras mentioned in Section~\ref{sec:IsomGroup}. In these cases the isometry groups act transitively on the spheres $\mathcal S_{\mathbb K}(0,r_1,r_2)$, $\mathbb K=\mathbb R, \mathbb C$, or $\mathbb Q$. This allows to define the measure on the Grassmannians. Moreover, the transitive action permits to realise the Grassmannians as orbit spaces under the action of the isometry groups. 
Note that the spaces $\mathbb R^n$, $\mathbb C^n$, and $\mathbb Q^n$ are Abelian algebras with respect to the summation. Any $k$ dimensional vector subspace $V$ is a subalgebra that has $n-k$ dimensional orthogonal complement that is also a subalgebra. 
This property is not trivial for the non-commutative subalgebras and therefore we restrict ourself to the consideration of {\it orthogonally complemented} Grassmannians.

\subsubsection{Orthogonally complemented Grassmannians}

A subalgebra $V\subset \h^n_{\mathbb K}$ is called {\it homogeneous} if it is invariant under the action of dilation~\eqref{eq:delta}.
In the following definition we use the inner product from the definition of $H$-type Lie algebra $\h=(\h_1\oplus\h_2,[.\,,.],\langle.\,,.\rangle_\mathbb R)$. 
\begin{definition} We say that a homogeneous subalgebra $V\subset \h^n_{\mathbb K}$ is an orthogonally complemented homogeneous subalgebra of $\h^n_{\mathbb K}$ if the orthogonal complement
$V^\perp$ is a homogeneous subalgebra  of $\h^{n}_{\K}$.  
\end{definition}

\begin{lemma}\label{lem: orth.compl.Heis} Let $V\subset \h^{n}_{\K}$ be an orthogonally complemented  homogeneous subalgebra. Then:
\begin{itemize}
\item[i)] In the case $\h^{n}_{\K}$ for $\K=\R,\HH$ we have
\begin{itemize} 
\item[i-1)]if $\dim_{\R} V \le n$, then $V\subset \h_1$ (and hence $V$ is commutative);
\item[i-2)]if $\dim_{\R} V > n$, then $\h_2 \subset V$.
\end{itemize}
\item[ii)] In the case $\h^{n}_{\CC}$ we have
\begin{itemize} 
\item[ii-1)]if $\dim_{\R} V \le 2n$, then $V\subset \h_1$ (and hence $V$ is commutative);
\item[ii-2)] if $\dim_{\R} V > 2n$, then $\h_2 \subset V$. 
\end{itemize}
\end{itemize}
\end{lemma}

\begin{proof} We start from the case $\h^{n}_{\R}$. We sketch the construction of the complementary homogeneous subalgebras. The horizontal vector space $\h_1$ is a symplectic vector space with the symplectic form~\eqref{eq:symplectic-form}, see Section~\ref{IsohR}. Let $W\subset h_1$ be a vector space of dimension $d=\dim_{\R}(W)\geq \frac{1}{2}\dim_{\mathbb R}(\h_1)=n$. It was shown in~\cite[Lemma 3.26]{FrSerSC07}, that there is a vector space $W'\subset \h_1$ such that $W'\oplus W=\h_1$ and $\omega(v,w)=0$ for all $v,w\in W'$. The relation between the symplectic form and the commutation relation shows that $W'$ is a commutative subalgebra of $\h_1$. We choose an orthonormal basis $\{e_1,\ldots,e_{d'}\}$, $d'=\dim(W')\leq n$ for $W'$ and extend it to an orthonormal basis 
$$
\{e_1,\ldots,e_{d'}, e_{d'+1},\ldots,e_{n}, f_1=J_{\epsilon}e_1,\ldots,f_n=J_{\epsilon}e_n\}.
$$
We denote 
$$V=W',\quad 
V'=\spn\{e_{d'+1},\ldots,e_{n}, f_1=J_{\epsilon}e_1,\ldots,f_n=J_{\epsilon}e_n\}\oplus\spn\{\epsilon\}.
$$
Then it is easy to see that $V\oplus V'$ are orthogonally complemented subalgebras, satisfying hypothesis of Lemma~\ref{lem: orth.compl.Heis}.
Note that the construction above gives all possible complementary subalgebras. 

We turn to $\h^n_{\CC}$ and will consider $\h_1\subset \h^n_{\CC}$ as a complex symplectic space with the complex symplectic form $\omega^{\CC}$. We have $\dim_{\mathbb C}(\h_1)=2n$ and $n$ is even. We can show that for any $W$ such that $\dim_{\CC}(W)\geq n$ there is a vector space $W'\subset \h_1$ satisfying $W'\oplus W=\h_1$ and $\omega^{\mathbb C}(v,w)=0$ for all $v,w\in W'$. The arguments are the same as at the beginning of the proof, since the arguments do not depend on the choice of the fields $\R$ or $\CC$, but only on the construction of the basis. 

Thus, by the construction, the vector space $W'$ is a complex isotropic vector space with respect to the complex symplectic form $\omega^{\CC}$ with $\dim_{\mathbb C}W'\leq n$. It contains the isotropic subspace $\tilde W'$ with respect to the real symplectic form $\omega$ and the complex span of $\tilde W'$ coincides with $W'$. The vector space $\tilde W'$ is a commutative real subalgebra and its complexification $W'$ will be a commutative complex subalgebra of the complex Heisenberg algebra $\h^n_{\CC}$. By making use notation~\eqref{eq:renotation} we conclude that $\dim_{\mathbb R}W'\leq 2n$, and $W'$ is a commutative subalgebra of $\h^1_{\CC}$ considered as a real Lie algebra. 

We consider two cases: $\dim_{\mathbb R}W'=2k\leq n$ and $\dim_{\mathbb R}W'=2k> n$. Let $\dim_{\mathbb R}W'=2k\leq n$. We find a real commutative orthonormal basis $\{e_1,\ldots, e_{2k}\}$ for $W'$ and extend it to an orthonormal basis $\{e_1,\ldots, e_{n}\}$. Then we denote
$
V=W'$, and set
$$
V'=\spn_{\R}\{e_{2k+1},\ldots, e_{n}, J_{\epsilon_1}e_i,J_{\epsilon_2}e_i,
J_{\epsilon_2}J_{\epsilon_1}e_i;\ i=1,\ldots,n\}\oplus\spn_{\R}\{\epsilon_1,\epsilon_2\}.
$$

In the case $\dim_{\mathbb R}W'=2k> n$, we choose orthonormal vectors $\{e_1,\ldots, e_{n}\}$ in $W'$ and extend them to the orthonormal basis 
$$
\{e_1,\ldots, e_{n}, J_{\epsilon_2}J_{\epsilon_1}e_1,\ldots,J_{\epsilon_2}J_{\epsilon_1}e_n\}
$$
of the maximal commutative subalgebra of $\h^1_{\CC}$. Without loss of generality we can assume that 
$$
\{e_1,\ldots, e_{n}, J_{\epsilon_2}J_{\epsilon_1}e_1,\ldots,J_{\epsilon_2}J_{\epsilon_1}e_p\},\quad n+p=2k
$$
is an orthonormal basis for $W'$. Now we denote $V=W'$ and set
$$
V'=\spn_{\R}\{J_{\epsilon_1}e_i,J_{\epsilon_2}e_i;\ i=1,\ldots,n,\ 
J_{\epsilon_2}J_{\epsilon_1}e_{p+1},\ldots, J_{\epsilon_2}J_{\epsilon_1}e_{n}\}\oplus\spn_{\R}\{\epsilon_1,\epsilon_2\}.
$$
We recall that 
$$
\langle e_i, J_{\epsilon_1}J_{\epsilon_2}e_i\rangle_{\mathbb R}=-\langle J_{\epsilon_1}e_i, J_{\epsilon_2}e_i\rangle_{\mathbb R}=\langle\epsilon_1,\epsilon_2\rangle_{\R}\langle e_i,e_i\rangle_{\R}=0
$$
because of the orthogonality of vectors $\epsilon_1,\epsilon_2$. Also  
$
\langle e_i, J_{\epsilon_1}J_{\epsilon_2}e_j\rangle_{\mathbb R}=0$ for $i\neq j$
because of the orthogonal decomposition~\eqref{eq:direct-product}.

The last case concerns with $\h^n_{\HH}$. We start from lower dimensional subalgebra $\h^1_{\HH}$. We choose a vector $v\in \h_1\subset \h^1_{\HH}$ with $\langle v,v\rangle_{\mathbb R}=1$ and define an orthonormal basis of the real Heisenberg algebra
$$
X_1=v,\quad X_2=J_{\epsilon_1}v,\quad ,\epsilon_1
$$
We use $\epsilon_2$ and construct the symplectic complex space $\h_1^{\CC}$ as in the previous case. Then the constructed multidimensional complexified Heisenberg algebra satisfies the commutation relations of the first line in~\eqref{eq:commutator-quaternion}. As in the previous case we find a space $W'$ such that $\omega^{\mathbb C}(u,v)=0$ for all $u,v\in W'$. Since $W'$ is a complex vector space it has even real dimension $\dim_{\mathbb R}(W')=2k\leq 2n$ and therefore we can define one more real symplectic form on $W'$ (considered as a real vector space) by
$$
\omega_{3}(u_1,u_2)=\langle J_{\epsilon_3}u_1,u_2\rangle_{\mathbb R},\quad u_1,u_2\in W'.
$$
Then we set $V'=W'\cap L(W')$, where $L(W')$ is the Lagrangian subspace of the real symplectic space $(W',\omega_3)$. We obtain $\dim_{\R}(V)=p\leq k\leq n$ and it is by construction a commutative subspace of $\h_{\HH}^1$. Now we choose an orthonormal basis 
$
\{e_1,\ldots,e_p\}$ for $V$
and complement it to an orthonormal basis 
\begin{equation}\label{eq:BBB}
\{e_1,\ldots,e_p,e_{p+1},\ldots,e_{n}\}. 
\end{equation}
In the last step we extend~\eqref{eq:BBB} to an orthonormal basis of $\h_1\subset\h_{\HH}^n$ by 
$$
\{e_1,\ldots,e_n,J_{\epsilon_l}e_{1},\ldots,J_{\epsilon_l}e_n;\ \l=1,2,3\}. 
$$
We have obtained the orthogonally complemented subalgebras
$$
V,\quad V'=\spn_{\R}\{e_{p+1},\ldots,e_n,J_{\epsilon_l}e_{1},\ldots,J_{\epsilon_l}e_n;\ \l=1,2,3\}\oplus\h_2
$$
satisfying the statement of Lemma~\ref{lem: orth.compl.Heis}.
\end{proof}

\begin{theorem}\label{transitively def}  The group $ \Iso(\h^{n}_{\K})$ acts transitively on the family of orthogonally complemented subalgebras of $\h^n_{\mathbb K}$,
i.e., if $V$ and $V'$ are subalgebras of the same dimension that have orthogonal subalgebras $V^\perp$ and $V'^\perp$, then there exists $\mc A \in \Iso(\h^{n}_{\K})$ such that
$$
V' = \mc A(V), \qquad V'^\perp = \mc A(V^\perp).
$$
\end{theorem}

\begin{proof} We start from $\h^{n}_{\R}$. Suppose first that $k=\dim_{\mathbb R} V=\dim_{\mathbb R} V'\le n$. Then both $V$ and $V'$ are commutative by Lemma~\ref{lem: orth.compl.Heis}. 
Let $\{z_1,\dots,z_k\}$ and $\{z'_1,\dots,z'_k\}$ be 
orthonormal bases of the corresponding $V$ and $V'$ with respect to the scalar product $\langle.\,,.\rangle_{\mathbb R}$. Since each of the bases belong to the isotropic space of the real symplectic form $\omega$ in~\eqref{eq:symplectic-form}, then the bases are orthonormal with respect to the Hermitian scalar profuct~\eqref{eq:complexHermitian}; that is
\begin{equation}\label{eq:Hermitian product}
\langle z_i\vert z_j \rangle_{\mathbb C}=\langle z_i,z_j \rangle_{\mathbb R}-i\omega(z_i,z_j )=0.
\end{equation} The same holds for $\{z'_1,\dots,z'_k\}$.

By the Gram-Schmit procedure for Hermitian scalar products the orthonormal families $\{z_1,\dots,z_k\}$ 
and $\{z'_1,\dots,z'_k\}$
can be extended to orthonormal bases $\mc Z=\{z_1,\dots,z_k, z_{k+1},\dots,z_n\}$ 
and  $\mc Z'=\{z'_1,\dots, z'_k, z'_{k+1},\dots,z'_n\}$ of $\h_1\subset\h^{n}_{\R}$. Then $\h_1$ spanned over $\mathbb C$ by $\mc Z$ and also by $\mc Z'$ is an $n$-dimensional complex space. We can find $A\in \U(n)$ such that $A(z_j)=z'_j$. Then
$\mc A= A\times \Id_{\h_2} \in \Iso(\h^{n}_{\R})$ and $A(V)=V'$. Since 
$$
\U(n)=\Orth(2n)\cap\GL(n,\CC)\cap\Sp(2n,\R)
$$
we conclude that $A \in \Orth(2n)$ and $\mc A\in \Orth(2n+1)$, and therefore
$\mc A(V^\perp)=V'^\perp$.  This completes the proof of the assertion when $k\le n$.

Suppose now $k>n$ and let $V$ and $V'$ be two orthogonally complemented  subalgebras of $\h^n_{\mathbb R}$.
The assertion follows by the previous arguments applied to the orthogonal complements $V^\perp$ and $V'^\perp$. 

Consider the Lie algebra $\h^{n}_{\CC}$. The complex dimension of $\h_1\subset \h^{n}_{\CC}$ is equal to $2n$. Let $V$ and $V'$ be orthogonally complemented subalgebras of complex dimension $k\leq n$. By the construction in Lemma~\ref{lem: orth.compl.Heis} the bases of $V$ or $V'$ will be also orthonormal bases with respect to the quaternion Hermitian product~\eqref{eq:quaternionHermitian}. 
Then we extend the bases of $V$ and $V'$ to bases of $\h_1\subset \h^{n}_{\CC}$ and apply the Gram-Schmidt procedure to make the bases orthonormal with respect to the quaternion Hermitian product. Noticing that 
$
\Sp(n)\cong \Sp(2n,\CC)\cap\U(2n),
$
we obtain that the bases will be orthogonal with respect to the original real scalar product and therefore will preserve the orthogonally complemented subalgebras. We finish the proof as in the previous case.

The last case is the Lie algebra $\h^{n}_{\Q}$. In this case we use the similar arguments noting that an orthonormal basis (with respect to $\langle.\,,.\rangle_{\mathbb R}$) for a commutative subalgebra $V$ will be orthogonal with respect to the quaternion Hermitian form~\eqref{eq:Hermitian-product1}. Therefore, the basis can be extended to an orthonormal basis for $\h^{n}_{\Q}$ with respect to the quaternion Hermitian form~\eqref{eq:Hermitian-product1}. The group $\Sp(n)$ acts transitively on a set of such kind of extended bases for $\h^{n}_{\Q}$, preserving the orthogonality. 
\end{proof}

\begin{definition}\label{grass}
The set $\Gr(k,\h^n_{\K})$, $1\le k\le \dim_{\R}\h^n_{\K}$ of orthogonally complemented homogeneous subalgebras of the same topological dimension $k$ is called the  Grassmannian of the Heisenberg algebra $\h^n_{\K}$.
\end{definition}
According to~\eqref{eq:form-of-isometry} we can write $\mc A\in \Iso(\h^{n}_{\K})$ as $\mc A=(
\mc U,\mc V)\subset \Orth(\h_1)\times \Orth(\h_2)$. If $V=H\otimes T$, $H\subset\h_1$, $T\subset \h_2$, is an orthogonally complemented subalgebra 
then the action of $\Iso(\h^{n}_{\K})$ on $\Gr(k,\h^n_{\K})$ is given by
\begin{equation}\label{eq:action-algebra}
\mc A.V=\mc U.H\times \mc V.T.
\end{equation}
Theorem~\ref{transitively def} immediately implies the corollary.
\begin{corollary}\label{independent}  
The Grassmannian  $\Gr(k,\h^n_{\K})$ is the orbit of the action of the isometry group $\Iso(\h^{n}_{\K})$ issued from an orthogonally complemented homogeneous $k$-dimensional subalgebra $\hat V\subset \h^n_{\K}$.
\end{corollary}

\begin{example}
Consider the Heisenberg algebra $\h^2_{\mathbb R}$ with the basis~\eqref{eq:HeisFirstTime} and the commutation relations~\eqref{eq:HeisCommFirstTime}.
We write
$
\hat V=\spn_{\mathbb R}\{X_1,X_2\}.
$
The space $\hat V$
is a commutative subalgebra of $\h^2_{\mathbb R}$ orthogonally complemented by the commutative subalgebra $V=\spn_{\mathbb R}\{Y_1,Y_2,\epsilon\}$.
The isometry group $\Iso(\h^2_{\mathbb R})$ is given by~\eqref{eq:form-of-isometry} with $\mathbb A=\U(2)$.
The orbit of the isometry group $\Iso(\h^2_{\mathbb R})$ acting on $\hat V$ is the Grassmannian $\Gr(2,\h^2_{\mathbb R})$. The planes in this Grassmannian intersect in one point and it is analogue of the Grassmann manifold of 2 dimensional planes in $\mathbb R^4\cong\h_1\subset\h^2_{\mathbb R}$. The orbit of the isometry group $\Iso(\h^2_{\mathbb R})$ acting on the complementary subalgebra $V$ is the Grassmannian $\Gr(3,\h^2_{\mathbb R})$. The planes in $\Gr(3,\h^2_{\mathbb R})$ intersect in the stright line coinciding with the center of $\h^2_{\mathbb R}$.
\end{example}
In the following example we show that the complementary subalgebras can both contain the elements of the center and be both commutative.
\begin{example}
Let us consider the Lie algebra $\h^1_{\mathbb C}$ with the basis~\eqref{eq:Basis-1} and the commutation relations~\eqref{eq:ComhC}. We set
$$
\hat V=\spn_{\mathbb R}\{X_1,X_4,\epsilon_1\}=\spn_{\mathbb C}\{X_1\}\oplus\spn_{\mathbb R}\{\epsilon_1\}.
$$
It is a commutative subalgebra of $\h_{\mathbb C}$ that is orthogonally complemented by the commutative subalgebra $V=\spn_{\mathbb R}\{X_2,X_3,\epsilon_2\}$. The isometry group $\Iso(\h^1_{\mathbb C})$ is given by~\eqref{eq:form-of-isometry} with $\mathbb A=\Sp(1)=\Sp(2,\mathbb C)\cap\U(2)$.
The orbit of the isometry group $\Iso(\h^1_{\mathbb C})$ acting on $\hat V$ is the Grassmannian $\Gr(3,\h^1_{\mathbb C})$. The planes in $\Gr(3,\h^1_{\mathbb C})$ intersect in one point.
\end{example}


\subsubsection{Grassmannians $\Gr(k, \h^n_{\K})$ as quotient spaces}


Let us fix an orthogonally complemented subalgebra $\hat V\subset \h^n_{\K}$ and write
\begin{equation*}
 K(\hat V)=
\Big\{\mc A\in\Iso(\h^n_{\K}):\ \mc A. \hat V=\hat V\Big\}
\end{equation*}
for the isotropy group of $\hat V$.
The canonical projection
\begin{equation}\label{eq:Pi-1}
\Pi\colon \Iso(\h^n_{\K})\to \Iso(\h^n_{\K})/ K(\hat V)
\end{equation}
is a continuous map.
The action of $\Iso(\h^{n}_{\K})$ on $\Gr(k,\h^n_{\K})$ is transitive, by Lemma~\ref{transitively def}.  We identify the left cosets from $\Iso(\h^n_{\K})/ K(\hat V)$ with elements in $\Gr(k, \h^n_{\K})$ by
\begin{equation}\label{eq:Gamma-1}
\begin{array}{ccccccc}
\Gamma\colon& \Iso(\h^n_{\K})/K(\hat V)&\to& \Gr(k, \h^n_{\K})
\\
&\mc A.\, K(\hat V)&\mapsto& V=\mc A.\hat V.
\end{array}
\end{equation}
The map $\Gamma$ is a diffeomorphism, see~\cite[Theorem 3.62]{Warner94}. 


\subsubsection{Measure on the Grassmannians $\Gr(k, \h^n_{\K})$}\label{sec:MeasureGrassmannian}


The groups $\Pin(\h_2,\langle.\,,.\rangle)$, $\h_2\subset\h^n_{\mathbb K}$  and $\mathbb A$ from Section~\ref{sec:IsomGroup} are compact Lie groups and therefore they carry normalised Haar measures that we will denote by $\lambda_{\Pin}$ and $\lambda_{\mathbb A}$, respectively. We also will denote $\lambda=\lambda_{\Pin}\times\lambda_{\mathbb A}$ the normalised product measure on the space $\mathbb P=\Pin(\h_2,\langle.\,,.\rangle)\times \mathbb A$. The map $\phi\colon \mathbb P\to\Iso(\h^n_{\mathbb K})$ from~\eqref{eq:form-of-isometry} mapping $\theta\in\mathbb P\to \phi(\theta)=\mc A=(
\mc U,\mc V)\in\Iso(\h^n_{\mathbb K})$ is continuous surjective map that makes possible to push forward the measure $\lambda$ to $\Iso(\h^n_{\mathbb K})$. Then the map $\Psi=\Gamma\circ \Pi\circ\phi$, where $\Pi$ and $\Gamma$ are defined in~\eqref{eq:Pi-1} and~\eqref{eq:Gamma-1} respectively, allows us to push forward the normalised Haar measure $\lambda$ from $\mathbb P$ to $\Gr(k, \h^n_{\K})$. We say that a set $\Omega\subset \Gr(k, \h^n_{\K})$ is measurable if $\Psi^{-1}(\Omega)\subset\mathbb P$ is measurable with respect to the measure $\lambda$. The measure $\mu$ on $\Gr(k, \h^n_{\K})$ is defined by 
$$
\mu(\Omega)=(\Psi_{\sharp}\lambda)(\Omega)=\lambda(\Psi^{-1}(\Omega)\big)=\lambda\Big\{\theta\in\mathbb P:\ V=\phi(\theta).\hat V=\mc A.\hat V\in \Omega\Big\},
$$
for any measurable $\Omega\subset \Gr(k, \h^n_{\K})$. We express the push forward in the integral form
\begin{equation}\label{eq:Gr1-1}
\int_{\Gr(k, \h^n_{\K})}f(V)\,d\mu(V)=\int_{\mathbb P}f(\phi(\theta).\hat V)\,d\lambda(\theta)
\end{equation}
for any measurable function $f$ on the Grassmannian.


\subsubsection{The groups $\Iso(\h^n_{\K})$ and the product of spheres}


According to Lemmas~\ref{lem:transit-hr} and~\ref{lem:transit-hr-1} the groups $\Iso(\h^n_{\K})$ act transitively on the product of two spheres
$$
\mathcal S_{\K}(0,r_1,r_2)=S^{h}(0,r_1)\times S^{v}(0,r_2)\subset\h_1\oplus \h_2\subset\h^n_{\K}.
$$
with
$$
S^{h}(0,r_1)=\{g=(x,0)\in \h^n_{\K}:\ \|x\|_E=r_1\},\quad
$$
$$
S^{v}(0,r_2)=\{g=(0,t)\in \h^n_{\K}:\ \|t\|_E=r_2\}.
$$
The group $\Iso(\h^n_{\K})$ acts on $\mathcal S_{\K}(0,r_1,r_2)$ by the following
$$
 \mc A.(y,w)=(\mc Uy,\mc Vw)\quad\text{for any}\quad
  \mc A=(\mc U,\mc V)\in \Iso(\h^n_{\K}),\ \ (y,w)\in \mathcal S_{\K}(0,r_1,r_2).
$$
We fix $(x,t)\in \mathcal S_{\K}(0,r_1,r_2)$ and define the isotropy subgroups
\begin{equation}\label{eq:IsotropyPoint}
\begin{split}
K^h_{(x,t)}=\{\mc A\in\Iso(\h^n_{\K}):\  \mc A.(x,t)=(\mc Ux,\mc Vt)=(x,\mc Vt)\},
\\
K^v_{(x,t)}=\{\mc A\in\Iso(\h^n_{\K}):\  \mc A.(x,t)=(\mc Ux,\mc Vt)=(\mc Ux,t)\}.
\end{split}
\end{equation}
We can realise both spheres as homogeneous spaces
under the action of the respective groups, see~\cite[Theorem 3.62]{Warner94}. Namely, we write
$$
\Pi\colon \Iso(\h^n_{\K})\to \Iso(\h^n_{\K})/\Big(K^h_{(x,t)}\times K^v_{(x,t)} \Big)\cong S^{h}(0,r_1)\times S^{v}(0,r_2).
$$ 
We will use the projections
$$
\Pi^h\colon \Iso(\h^n_{\K})\to S^{h}(0,r_1), \quad
\Pi^v\colon \Iso(\h^n_{\K})\to S^{v}(0,r_2).
$$
The map $\phi$ from~\eqref{eq:form-of-isometry} is continuous and surjective. It allows us  to define the push forward measures
$$
\mu^h=(\Pi^h\circ\phi)_{\sharp}\lambda_{\Pin}\quad\text{and}\quad 
\mu^v=(\Pi^v\circ\phi)_{\sharp}\lambda_{\mathbb A}.
$$

\begin{lemma}\label{lem:PushSphere}
The measures $\mu^h$ and $\mu^v$ are normalised measures on the spheres $S^{h}(0,r_1)$ and $S^{v}(0,r_2)$, respectively. Moreover,
$$
\int_{S^{v}(0,r_2)}d\mu^v(w)\int_{S^{h}(0,r_1)}f(y,w)d\mu^h(y)=
\int_{\mathbb P}f\big(\phi(\theta).(x,t)\big)\,d\lambda(\theta)
$$
for any measurable function $f$ on $S_{\K}(0,r_1,r_2)$ and the isotropy point $(x,t)$ from~\eqref{eq:IsotropyPoint}.
\end{lemma}
\begin{proof}
The transitive action of the isometry group on $S^{h}(0,r_1)$ and $S^{v}(0,r_2)$ ensures that the measures $\mu^h$ and $\mu^v$ are uniformly distributed on the respective spheres. Therefore, they are the spherical measures up to constants, see~\cite{MR1333890}.

Let $C^h\subset S^{h}(0,r_1)$, $C^v\subset S^{v}(0,r_2)$ be measurable sets and let $C\subset\mathbb P$ be its preimage under the map $\Pi\circ \phi=(\Pi^h\circ \phi,\Pi^v\circ \phi)$ from~\eqref{eq:form-of-isometry}. We write $\theta=(\alpha, A)\in \Pin(\h_2,\langle.,\,.\rangle)\times\mathbb A$ and $\phi(\theta)=\phi(\alpha, A)=(\mc U,\mc V)\in \Iso(\h^n_{\K})$, where $\mc U=J_\alpha\circ A$ and $\mc V=\kappa(\alpha)$ then
\begin{eqnarray*}
&&\int_Cf\big(\phi(\theta).(x,t)\big)\,d\lambda(\theta)
\\
&=&
\int\limits_{\{\alpha\in \Pin:\ (\kappa(\alpha).t)\in C^v\}}d\lambda_{\Pin}(\alpha)\int\limits_{\{A\in\mathbb A:\ (J_{\alpha}\circ A.x)\in C^h\}}f(\mathcal U.x,\mc V.t)\,d\lambda_{\mathbb A}(A)
\\
&=&
\int\limits_{\{\alpha\in \Pin:\ (\kappa(\alpha).t\in C^v\}}d\lambda_{\Pin(\alpha)}\int\limits_{\{A\in\mathbb A:\ (A.x)\in C^h\}}f(\mathcal U.x,\mc V.t)\,d\lambda_{\mathbb A}(A)
\\
&=&
c\int\limits_{C^v}d\mu^v(w)\int\limits_{C^h}f(y,w)\,d\mu^h(y).
\end{eqnarray*}
In the third line we used the fact that the group $\mathbb A$ is already acts transitively on the spheres $S^{h}(0,r_1)$ and therefore the push forward measure $\mu^h=(\Pi^h\circ\phi)_{\sharp}\lambda_{\Pin}$ is up to a constant the Hausdorff spherical measure of $S^{h}(0,r_1)$ due to the fact that both measures will be uniformly distributed on $S^{h}(0,r_1)$.
\end{proof}

\begin{remark}
Consider the subgroup $\mathbb A\times \Id_{\h_2}\subset \Iso(\h^n_{\K})$. Since $\mathbb A\times \Id_{\h_2}$ leaves the center $\h_2$ invariant we can define the action of $\mathbb A$ only on the horizontal slot of coordinates. Since the group $\mathbb A$ acts transitively on the spheres $S^{h}(0,r)$, the sphere $S^{h}(0,r)$ passing through the point $x\in\h_1$ with $\|x\|_E=r$ is a homogeneous manifold realised as a quotient of $\mathbb A$ by a subgroup fixing $x$. Then the integral form of the push forward of a measure $\tilde\lambda$ from $\mathbb A$ is the following
\begin{equation}\label{eq:Gr1-2}
\int_{S^{h}(0,r)}f(y)\,d\mu^h(y)=\int_{\mathbb A}f(\mc Ux)\,d\tilde\lambda(\mc U),\quad x\in S^{h}(0,r),
\end{equation}
for any measurable function $f$ on $S^{h}(0,r)$.
\end{remark}

%
%

\section{Integral formula on ``special'' $H$-type algebras}



\subsection{Overview of the formula in $\mathbb R^n$}

For the Grassmann manifolds in the Euclidean space the following formula is known~\cite{MR105724}.
Let $V\in\Gr_k(\mathbb R^n)$ and 
$$
F(V)=\int_V f(x)d\sigma(x),
$$
where $f$ is a non-negative measurable function in $\mathbb R^n$ and $\sigma$ is the $k$-dimensional Lebesgue measure on the plane $V$. Then
\begin{equation}\label{eq:main}
\int_{\Gr_k(\mathbb R^n)}F(V)\, d\mu(V)=\frac{m(S^{k-1}(0,1))}{m(S^{n-1}(0,1))}\int_{\mathbb R^n}\|x\|^{k-n}_Ef(x)\,dx,
\end{equation}
where $\mu$ is a normalised invariant under the rotational group measure on the Grassmann manifold, $m(S^{k-1}(0,1))$ is the measure of the unit sphere $S^{k-1}(0,1)\subset\mathbb R^k$, $dx$ is the Lebesgue measure on $\mathbb R^n$, and $\|x\|_E$ is the Euclidean norm of $x$. Our aim is finding an analogous expression for the three types of the Heisenberg algebras mentioned in Section~\ref{sec:IsomGroup}.


\subsection{Formula for special $H$-type Lie algebras}


We start from the case of orthogonally complemented Grassmannians $\Gr(k, \h^n_{\K})$ which elements consist of the commutative subalgebras and do not include elements of the center. We call them shortly {\it "horizontal" Grassmannians}. In this case we recover formula~\eqref{eq:main}.

\subsubsection{Formula for the "horizontal" Grassmannians }


We consider both manifolds $\Gr(k, \h^n_{\K})$ and $S^{h}(0,\|x\|)\subset \h^n_{\K}$ as homogeneous subspaces under the action of the subgroup $\mathbb A\times\Id\subset\Iso(\h^n_{\K})$. Here $\|x\|$ is the Euclidean, or Hermitian norm on $\h_1\subset \h^n_{\K}$, accordingly to $\mathbb K$.

Let $\Gr(k, \h^n_{\K})$ be the Grassmannian consisting of orthogonally complemented commutative subalgebras that do not contain elements of the center. In this case the topological and homogeneous dimensions of $V\in \Gr(k, \h^n_{\K})$ coincide and we denote them by ${\bf d_m}={\bf d_t}=k$.  
Let also $\mathcal L^{k}$ denote $k$-dimensional Lebesgue measure on a generic plain $V\in \Gr(k, \h^n_{\K})$ with $V\subset\h_1$, $\h_1\subset \h^n_{\K}$. Set
\begin{equation}\label{eq:defF-1}
F(V)=\int_{V}f(y)\,d\mathcal L^{k}(y),\quad y\in V\subset\h_1.
\end{equation}
Here $f\colon \h_1\to \mathbb R$ is a non-negative measurable function.

\begin{theorem}\label{th:main} The formula
\begin{eqnarray*}
\int_{\Gr(k, \h^n_{\K})}F(V)\,d\mu(V)&=&
\int_{\Gr(k, \h^n_{\K})}\,d\mu(V)\int_{V}f(y)\,d\mathcal L^{k}(y)
\\
&=&
C
\int_{\h_1}\|z\|^{k-m_1}f(z)d\mathcal L^{m_1}(z),
\end{eqnarray*}
holds for any measurable non-negative function $f\colon \h_1\to\mathbb R$ and an orthogonally complemented Grassmannian $\Gr(k, \h^n_{\K})$ of commutative subalgebras that do not contain elements of the center of $\h^n_{\K}$. Here $\mathcal L^{m_1}$, $m_1=\dim_{\mathbb R}(\h_1)$ is the Lebesgue measure on $\h_1\subset \h^n_{\K}$ and $C>0$ is a constant.
\end{theorem}

\begin{proof}
Note first that $V=A.\hat V$ for any $V\in\Gr(k, \h^n_{\K})$ and some $A\in\mathbb A$. Therefore
$$
F(V)=F(A.\hat V)=\int_{A.\hat V}f(y)\,d\mathcal L^{k}(y)=\int_{\hat V}f(Ax)\,d\mathcal L^{k}(x)
$$
\begin{eqnarray}\label{eq:part1}
\int_{\Gr(k, \h^n_{\K})}&&F(V)\,d\mu(V)
=
\int_{\Gr(k, \h^n_{\K})}F(A.\hat V)\,d\mu(A.\hat V)\nonumber
=
\int_{\mathbb A}F(A.\hat V)\,d\lambda(A)\nonumber
\\
&=&
\int_{\mathbb A}\,d\lambda(A)\int_{\hat V}f(Ax)\,d\mathcal L^{k}(x)
=
\int_{\hat V}\,d\mathcal L^{k}(x)\int_{\mathbb A}f(Ax)d\lambda(A)
\end{eqnarray}
by~\eqref{eq:Gr1-1}.
\end{proof}

Let us consider the last integral, where $x\in \hat V$ will be also considered as a point on the  sphere $S^{h}(0,\|x\|)\subset \hat V\subset\h_1$. For that $x\in S^{h}(0,\|x\|)$ we can consider $S^{h}(0,\|x\|)$ as a homogeneous manifold under the action of $\mathbb A$ with the isotropy group that fixes $x$. We denote that sphere by $S^{h,x}(0,\|x\|)$, emphasising the fixed point on the sphere. Then we use the push forward $\mu^h$ of the normalized measure $\lambda$ from $\mathbb A$ to $S^{h,x}(0,\|x\|)$ and obtain
\begin{equation}\label{eq:part2}
\int_{\mathbb A}f(Ax)d\lambda(A)=\int_{S^{h,x}(0,\|x\|)}f(z)d\mu^h(z)=\tilde C\int_{S^{h}(0,1)}f(\|x\|\xi)d\mc S^{m_1-1}(\xi),
\end{equation}
where $d\mc S^{m_1-1}(\xi)$ is the surface measure on the unit sphere $S^{h}(0,1)$. In the last step we used the following calculations
\begin{eqnarray*}
&&\int_{S^{h,x}(0,\|x\|)}f(z)d\mu^h(z)=
c\int_{S^{h,x}(0,\|x\|)}\|x\|^{1-m_1}f(\|x\|\xi)dS(\|x\|\xi)
\\
&=&
C\int_{S^{h}(0,1)}\|x\|^{1-m_1}\|x\|^{m_1-1}f(\|x\|\xi)d\mc S^{m_1-1}(\xi)
=
C\int_{S^{h}(0,1)}f(\|x\|\xi)d\mc S^{m_1-1}(\xi),
\end{eqnarray*}
where $dS(\|x\|\xi)$ is the surface measure on the sphere $S^{h,x}(0,\|x\|)$, $x\in \hat V$.
Substituting integral~\eqref{eq:part2} into~\eqref{eq:part1}, we obtain (for $\rho=\|x\|$)
\begin{eqnarray}\label{eq:part3}
&&\int_{\hat V}d\mathcal L^{k}(x)\int_{S^{h}(0,1)}f(\|x\|\xi)d\mc S^{m_1-1}(\xi)\nonumber
\\
&=&
\int_{0}^\infty\rho^{k-1}d\rho\underbrace{\int_{S^{k-1}(0,1)}d\mc S^{k-1}(\zeta)}_{\text{constant}}\int_{S^{h}(0,1)}f(\|x\|\xi)d\mc S^{m_1-1}(\xi)\nonumber
\\
&=&
\hat C\int_{0}^\infty\rho^{k-1-(m_1-1)}d\rho\int_{S^{h}(0,1)}f(\rho\xi)\rho^{m_1-1}d\mc S^{m_1-1}(\xi)
\\
&=&
\hat C\int_{\h_1}\|z\|^{k-m_1}f(z)d\mathcal L^{m_1}(z).
\nonumber
\end{eqnarray}

\subsubsection{Formula for the "vertical" Grassmannians
}


Let $\Gr(k, \h^n_{\K})$ be a Grassmannian, where a typical orthogonally complemented subalgebra $V\in \Gr(k, \h^n_{\K})$ contains a non-trivial part of the center of $\h^n_{\K}$. 
Let $\hat V=\hat V_h\oplus \hat V_v$, $\hat V_v\neq\{0\}$ be an orthogonally complemented subalgebra, such that $\hat V_h\subset \h_1\subset \h_{\K}^n$ and $\hat V_v\subset \h_2\subset \h_{\K}^n$. We write $k_v=\dim(\hat V_h)$, $k_v=\dim(\hat V_v)$ for the topological dimensions  of the vector spaces $\hat V_h$ and $\hat V_v$. Thus $k=\dim \hat V=k_h+k_v$ is the topological dimension of orthogonally complemented subalgebra $ \hat V$. A generic element $V\in \Gr(k, \h^n_{\K})$ is the image of $\hat V$ under the action of $\Iso(\h^n_{\K})$. We write $y=(x,t)\in V\in \Gr(k, \h^n_{\K})$. Let 
$$
F(V)=\int_{V}f(x,t)\,d\mathcal L^{k}(x,t),
$$ 
where $f\colon \h^n_{\K}\to\mathbb R$ is a measurable non-negative function and $\mathcal L^k$, $k=\dim_{\mathbb R}(V)$ is the Lebesgue measure on $V$. We denote $m_1=\dim_{\mathbb R}(\h_1)$, $m_2=\dim_{\mathbb R}(\h_2)$, the topological dimensions of the horizontal $\h_1$ and vertical $\h_2$ layers 
of $\h^n_{\K}=\h_1\oplus \h_2$. Thus $N=m_1+m_2$ is the topological dimension of the Lie algebra $\h^n_{\K}$. Moreover, $\mathcal L^{m_1}$ and $\mathcal L^{m_2}$ are the respective Lebesgue measures on the vector spaces $\h_1$ and $\h_2$. 

\begin{theorem}\label{th:main-vertical}
The formula
\begin{eqnarray*}
\int_{\Gr(k, \h^n_{\K})}F(V)\,d\mu(V)&=&
\int_{\Gr(k, \h^n_{\K})}\,d\mu(V)\int_{V}f(x,t)\,d\mathcal L^{k}(x,t)
\\
&=&
C
\int\limits_{\mathbb R^{m_1}\times\mathbb R^{m_2}}\|x\|^{k_h-m_1}\|t\|^{k_v-m_2}f(x,t)\,d\mathcal L^{m_1}(x)d\mathcal L^{m_2}(t).
\end{eqnarray*}
holds for any measurable non-negative function $f\colon \h^n_{\K}\to\mathbb R$ and an orthogonally complemented Grassmannian $\Gr(k, \h^n_{\K})$ of subalgebras that contain a nontrivial element of the center of $\h^n_{\K}$. Here $C>0$ is a constant.
\end{theorem}
\begin{proof}
The isometry group $\Iso(\h^n_{\K})$ does not act transitively on the spheres with respect to any of the metrics $(D_2)-(D_4)$. This fact does not allow to obtain a uniformly distributed measure on a sub-Riemannian sphere by pushing forward the measure from the isometry group $\Iso(\h^n_{\K})$. Nevertheless, the transitive action of $\Iso(\h^n_{\K})$ on the product of spheres allows to prove Lemma~\ref{lem:PushSphere}. In the product of the integrals
\begin{equation}\label{eq:SvSh}
\int_{\mathbb P}f\big(\phi(\theta).(x,t)\big)\,d\lambda(\theta)=\int_{S^{v}(0,r_2)}d\mu^v(w)\int_{S^{h}(0,r_1)}f(y,w)d\mu^h(y)
\end{equation}
the measures $d\mu^v$ and $d\mu^h$ are the normalised measures on the spheres $S^{v}(0,r_2)$ and $S^{h}(0,r_1)$, respectively.  We can use successive independent dilations in the vertical and horizontal variables and write right-hand side of~\eqref{eq:SvSh} as a product of integrals with respect to the Hausdorff measures on the unit spheres 
\begin{equation}\label{eq:SvShUn}
C\int_{S^{v}(0,1)}d\mc S^{m_2-1}(\eta)\int_{S^{h}(0,1)}f(r_2\xi,r_1\eta)\,d\mc S^{m_1-1}(\xi).
\end{equation}
Futhermore, by making use of the polar coordinates in each of the vector spaces $\hat V_h$ and $\hat V_v$ we obtain
\begin{eqnarray*}
&&\int_{\hat V}d\mathcal L^{{\bf d_t}}(x,t)=\int_{\hat V_h}d\mathcal L^{k_h}(x)\int_{\hat V_v}d\mathcal L^{k_v}(t)
\\
&=&
\int_0^{\infty}\rho^{k_h-1}d\rho\int\limits_{S^{k_h-1}(0,1)}d\mc S^{k_h-1}(\phi)
\int_0^{\infty}r^{k_v-1}d r\int\limits_{S^{k_v-1}(0,1)}d\mc S^{k_v-1}(\psi)
\\
&=&
\tilde C \int_0^{\infty}\rho^{k_h-1}d\rho\int_0^{\infty}r^{k_v-1}d r.
\end{eqnarray*} 
We recall that the measure $\mu(V)$ on $\Gr(k, \h^n_{\K})$ is the pushforward of the measure $\lambda(\theta)$ from the group $\mathbb P$. It allows us to write
\begin{eqnarray*}\label{eq:part1-1}
&&\int_{\Gr(k, \h^n_{\K})}F(V)\,d\mu(V)
=
\int_{\Gr(k, \h^n_{\K})}\,d\mu(V)\int_{V}f(x,t)\,d\mathcal L^{k}(x,t)
\\
&=&
\int_{\hat V}\,d\mathcal L^{k}(x,t)\int_{\mathbb P}f\big(\phi(\theta).(x,t)\big)\,d\lambda(\theta)\nonumber
\\
&=&
C\int_0^{\infty}\rho^{k_h-1}d\rho
\int_0^{\infty}r^{k_v-1}d r\nonumber
\int_{S^{v}(0,1)}d\mc S^{m_2-1}(\eta)\int_{S^{h}(0,1)}f(\rho\xi,r\eta)d\mc S^{m_1-1}(\xi)\nonumber
\\
&=&
C\int_0^{\infty}\rho^{k_h-m_1}d\rho\int_0^{\infty}r^{k_v-m_2}d r\nonumber
\\
&\times&
\int_{S^{v}(0,1)}d\mc S^{m_2-1}(\eta)\int_{S^{h}(0,1)}\rho^{m_1-1}r^{m_2-1}f(\rho\xi,r\eta)d\mc S^{m_1-1}(\xi)
\nonumber
\\
&=&
\int\limits_{\mathbb R^{m_2}\times\mathbb R^{m_1}}\|x\|^{k_h-m_1}\|t\|^{k_v-m_2}f(x,t)d\mathcal L^{m_1}(x)d\mathcal L^{m_2}(t).
\end{eqnarray*}
It finishes the proof.
\end{proof}


\subsection{Application of the integral formula}

Let $\mathbb G$ be one of the special Heisenberg-type Lie groups of the topological dimension $N$ and the homogeneous dimension $Q$ with the Lie algebra~$\h^{n}_{\mathbb K}$, such that $\dim(\h_1)=m_1$.

\begin{corollary}\label{cor:11}
Let $\Sigma\subset \Sigma^{({\bf d_t},{\bf d_m})}$ be a collection of intrinsic $({\bf d_t},{\bf d_m})$-Lipschitz graphs on $\mathbb G$. Suppose that all the graphs $S\in \Sigma$ contain a common point $g_0\in\mathbb G$. Then for ${\bf d_m}p\leq Q$, $p>1$, we have $M_p(\Sigma)=0$. In the case ${\bf d_t}={\bf d_m}$ if $p{\bf d_t}> m_1$, then there is a family $\Sigma$ of intrinsic Lipschitz graphs such that $M_p(\Sigma)\neq 0$.
\end{corollary}

\begin{proof}
The proof of the statement that ${\bf d_m}p\leq Q$ implies $M_p(\Sigma)=0$ is the proof of Theorem~\ref{prop:4}. To show the statement in the opposite direction, we consider the Grassmannian $\Gr(k, \h^n_{\K})$ consisting of orthogonally complemented commutative subalgebras that do not contain elements of the center. In this case the topological and homogeneous dimensions of $V\in \Gr(k, \h^n_{\K})$ coincide; that is ${\bf d_m}={\bf d_t}=k$.  

Consider $S=V\cap B^{h}(0,1)$, where $V\in \Gr(k, \h^n_{\K})$ and $B^{h}(0,1)\in\h_1$ is the Euclidean unit ball and assume by contrary, that $M_p(\Sigma)=0$, where 
$$\Sigma=\{S=V\cap B^{h}(0,1),\ V\in \Gr(k, \h^n_{\K})\}.
$$ Let $f\in L^p(\h_1,\mathcal L^{m_1})$, $m_1=\dim_{\mathbb R}(\h_1)$. Then from Theorem~\ref{th:main} and the H\"older inequality for $\frac{1}{p}+\frac{1}{q}=1$ we obtain
\begin{eqnarray*}
\int_{\Gr(k, \h^n_{\K})}\,d\mu(V)\int_{V\cap B^{h}(0,1)}f(y)\,d\mathcal L^{{\bf d_t}}(y)
&=&
C
\int_{\h_1\cap B^{h}(0,1)}\|x\|^{{\bf d_t}-m_1}f(x)d\mathcal L^{m_1}(x)
\\
&\leq&
C\|f\|_{L^p(\h_1)}\Big(\int_{0}^1r^{\frac{p{\bf d_t}-m_1}{p-1}-1}\Big)^{\frac{p-1}{p}}.
\end{eqnarray*}
Then since $p{\bf d_t}>m_1$ the last integral is finite. It implies 
\begin{equation}\label{eq:apphor}
\int_{V\cap B^{h}(0,1)}f(y)\,d\mathcal L^{{\bf d_t}}(y)<\infty
\end{equation}
for $\mu$-almost all $S=V\cap B^{h}(0,1)\in\Sigma$, that contradicts the assumption $M_p(\Sigma)=0$.
\end{proof}

For the following corollary we assume that $\Gr(k, \h^n_{\K})$ is a Grassmannian, where a typical orthogonally complemented subalgebra $V\in \Gr(k, \h^n_{\K})$ contains a non-trivial part of the center of $\h^n_{\K}$. In this case necessarily ${\bf d_t}<{\bf d_m}$.
Let $V=V_h\oplus V_v$ where $V_v\neq\{0\}$, be an orthogonally complemented subalgebra, such that $V_h\subset \h_1\subset \h_{\K}^n$ and $V_v\subset \h_2\subset \h_{\K}^n$. We write $k_h=\dim(V_h)$, $k_v=\dim(V_v)$, ${\bf d_t}=k=k_h+k_v$, for the topological dimensions  of the vector spaces $V_h$ and $V_v$, and $m_1=\dim(\h_1)$, $m_2=\dim(\h_2)$, the  topological dimensions of the horizontal and the vertical layers of $\h_{\K}^n$.

\begin{corollary}\label{cor:22}
Let $\Sigma\subset \Sigma^{({\bf d_t},{\bf d_m})}$ be a collection of intrinsic $({\bf d_t},{\bf d_m})$-Lipschitz graphs. Suppose that all the graphs $S\in \Sigma$ contain a common point $g_0\in\mathbb G$. Then for ${\bf d_m}p\leq Q$, $p> 1$,  we have $M_p(\Sigma)=0$. In the case ${\bf d_t}<{\bf d_m}$, if $pk_h> m_1$, $pk_v> m_2$ then there is a family $\Sigma$ of intrinsic Lipschitz graphs such that $M_p(\Sigma)\neq 0$. 
\end{corollary}
\begin{proof}
We argue as in the proof of Corollary~\ref{cor:22}. Consider $S=V\cap \Big(B^{h}(0,1)\times B^{v}(0,1)\Big)$, where $V\in \Gr(k, \h^n_{\K})$ and $B^{h}(0,1)\in\h_1$, $B^{v}(0,1)\in\h_2$ are the Euclidean unit balls. Asssume that $M_p(\Sigma)=0$, where 
$$
\Sigma=\{S=V\cap \Big(B^{h}(0,1)\times B^{v}(0,1)\Big),\ V\in \Gr(k, \h^n_{\K})\}.
$$ 
Let $f\in L^p(\h^n_{\K},\mathcal L^N)$, $N=\dim_{\mathbb R}(\h^n_{\K})$. Then from Theorem~\ref{th:main-vertical} and the H\"older inequality for $\frac{1}{p}+\frac{1}{q}=1$ we obtain
\begin{eqnarray*}
&&
\int_{\Gr(k, \h^n_{\K})}\,d\mu(V)\int_{V}f(x,t)\,d\mathcal L^{k}(x,t)
\\
&=&
C
\int\limits_{\mathbb R^{m_2}\cap B^v(0,1)\times\mathbb R^{m_1}\cap B^h(0,1)}\|x\|^{k_h-m_1}\|t\|^{k_v-m_2}f(x,t)d\mathcal L^{m_1}(x)d\mathcal L^{m_2}(t)
\\
&\leq&
C\|f\|_{L^p(\h^n_{\K})}\int\limits_{\mathbb R^{m_2}\cap B^v(0,1)\times\mathbb R^{m_1}\cap B^h(0,1)}\|x\|^{(k_h-m_1)\frac{p}{p-1}}\|t\|^{(k_v-m_2)\frac{p}{p-1}}d\mathcal L^{m_1}(x)d\mathcal L^{m_2}(t)
\\
&=&
\tilde C\|f\|_{L^p(\h^n_{\K})}\int\limits_0^1r^{\frac{pk_h-m_1}{p-1}-1}dr\int\limits_0^1\rho^{\frac{pk_v-m_2}{p-1}-1}d\rho<\infty
\end{eqnarray*}
Since $pk_h>m_1$ and $pk_v>m_2$, then 
$
\int_Vf(x,t)d\mathcal L^k(x,t)<\infty
$
for $\mu$-almost all plains $S\in\Sigma$, that contrudicts to the assumption $M_p(\Sigma)=0$. 
\end{proof}
\begin{remark}

\end{remark}
Two conditions 
$$
pk_h>m_1\quad\text{and}\quad pk_v>m_2
$$
imply
$$
p(k_h+k_v)=p{\bf d_t}>m_1+m_2=N,
\quad
p(k_h+2k_v)=p{\bf d_m}>m_1+2m_2=Q
$$
In general $p{\bf d_m}>Q$, $p>1$, does not imply both conditions $pk_h>m_1$ and $pk_v>m_2$ in spite that the second one for $\h^n_{\mathbb R}$ and $\h^n_{\mathbb Q}$ is always fulfilled for the Grassmannians $\Gr(k, \h^n_{\K})$ where a typical orthogonally complemented subalgebra $V$ contains a non-trivial part of the center. In both cases $\h^n_{\mathbb R}$ and $\h^n_{\mathbb Q}$ the subalgebra $V$ necessarily contains the entire center, and therefore $k_v=m_2$. From the other side we note that the Lipschitz surfaces meet each other not only at one point but at the entire center. 

In the case of $\h^n_{\mathbb C}$ and $V=V_h\oplus V_v$ with $k_h=\dim(V_h)$ and $k_v=\dim(V_v)=1$ we obtain
that 
$$
p{\bf d_m}=p(k_h+2)>Q=m_1+4\quad\Longrightarrow\quad pk_h>m_1+4-2p\geq m_1\quad\text{if}\quad p\leq 2
$$
but
$$
pk_v=p>m_2=2 \quad\text{if}\quad p> 2
$$
Thus even if the Lipschitz surfaces intersects in one point, our example does not give the answer to the question: {\it is there an example on a Carnot group that is not $\mathbb R^n$ where the condition $p{\bf d_m}\leq Q$ with ${\bf d_t}<{\bf d_m}$ is necessary for the system of Lipschitz surfaces intersecting in one point to be $p$-exceptional?}
\bibliographystyle{alpha}
\bibliography{Bibliography}

\end{document}